\documentclass[12pt]{amsart}
\usepackage{graphicx}
\usepackage{amsmath, amscd, amssymb, amsthm}
\usepackage[subnum]{cases}
\usepackage{changepage}
\usepackage{xcolor}
\usepackage{marginnote}
\usepackage{hyperref}
\usepackage{cleveref}
\usepackage[all]{xy}
\usepackage{tabularx}
\usepackage{ltablex}
\usepackage{longtable}
\usepackage[T1]{fontenc}

\newtheorem{theorem}{Theorem}[section]
\newtheorem{lemma}[theorem]{Lemma}

\newtheorem{proposition}[theorem]{Proposition}
\newtheorem{corollary}[theorem]{Corollary}
\newtheorem{conjecture}[theorem]{Conjecture}

\theoremstyle{definition}
\newtheorem{definition}[theorem]{Definition}
\newtheorem{remark}[theorem]{Remark}
\numberwithin{equation}{section}

\newtheorem{assumption}[theorem]{Assumption}
\newtheorem{setting}[theorem]{Setting}

\usepackage{fancyhdr}
\begin{document}

\normalfont

\title{Category and Cohomology of Hodge-Iwasawa Modules}
\author{Xin Tong}

\maketitle

\begin{abstract}
\rm In this paper we study the corresponding categories and the corresponding cohomologies of the Hodge-Iwasawa modules we developed in our series papers on Hodge-Iwasawa theory. The corresponding cohomologies will be essential in the corresponding development of the contact with the corresponding Iwasawa theoretic consideration, while they are as well very crucial in the corresponding study of the corresponding deformations of local systems over general analytic spaces. We contact with some applications in analytic geometry and arithmetic geometry which all have their own interests and deserve further study for us in the future, including local systems over general analytic spaces after Kedlaya-Liu, arithmetic Riemann-Hilbert correspondence in families after Liu-Zhu, and equivariant Iwasawa theory and geometrization of equivariant Iwasawa theory after Berger-Fourquaux and Nakamura.
\end{abstract}

\newpage

\tableofcontents

\newpage

\section{Introduction}

\subsection{The Main Scope of the Discussion}

\noindent The corresponding Hodge-Iwasawa theory was studied in our papers \cite{XT1} and \cite{XT2}. We in the corresponding deformed setting discussed the corresponding relative $p$-adic Hodge theory after \cite{KP}, \cite{KL15} and \cite{KL16}. We call the theory partially Iwasawa in the corresponding sense that the corresponding deformation will reflect the corresponding Iwasawa theoretic information which was observed as in \cite{KP}. On the Hodge-theoretic side, the corresponding cohomology in the deformed setting will also reflect the corresponding towers related to quotients (even nonabelian) of fundamental groups.\\

\indent Therefore in our current consideration in this paper, we apply all we have developed to the corresponding categories and the corresponding cohomologies. For instance, very interesting application would be the corresponding pro-\'etale cohomology of general local systems over some smooth proper rigid analytic spaces as in \cite{KL3}. Definitely we expect more applications in the corresponding Iwasawa theoretic consideration. \\

\indent The other interesting things we would like to pursue in the application of our previously developed theory is the corresponding categorical and $K$-theoretic study on the corresponding $(\varphi,\Gamma)$-modules over very relative Robba rings. The corresponding categories should form some well-established abelian ones which could further have more derived enrichment (carrying some six functors).\\

\indent The second goal of this paper is also to consider the Hodge-Structure on more general spaces. Although we will not systematically consider some general framework, since it will impossible to handle very general adic spaces. To be more general than the corresponding rigid analytic spaces, we are going to consider the corresponding spaces admitting atlas of the so-called $k_\Delta$ affinoid spaces as in \cite{TC} and \cite{DFN}, where $\Delta$ is some valuation group larger than $|k^\times|$.\\

\indent We briefly mention now the corresponding structures of the paper. In the second section we are going to study the corresponding Hodge-Iwasawa modules over logarithmic towers which generalizes for instance in the mixed-characteristic case the corresponding consideration in \cite[Chapter 7]{KL16}. Then in the third section we study after that the corresponding categories and cohomologies of the corresponding Hodge-Iwasawa modules involved. Then we are going to study the corresponding cohomology of more general analytic spaces. In fourth section we considere the Hodge-Iwasawa module over rigid analytic spaces, $k_\Delta$-analytic spaces and we contact the corresponding arithmetic Riemann-Hilbert correspondence after \cite{LZ}. The fifth section is for the corresponding equivariant Iwasawa theory of de Rham $(\varphi,\Gamma)$-modules which is a equivariant version of the corresponding Iwasawa theory in \cite{Nakamura1}. In such context, we still emphasize the corresponding geometrization aspect along \cite{KP} on the reconstruction effect of the corresponding Iwasawa deformation in different contexts. Note that the story is equivariant and relative inspired by Nakamura \cite{Nakamura3}, Kedlaya-Pottharst \cite{KP}, as well as some further motivation from \cite{BF1},\cite{BF2} and \cite{FK}.\\

\subsection{Results Involved}

\indent We searched many interesting directions in this paper to apply the Hodge-Iwasawa consideration in \cite{XT1} and \cite{XT2}. We present in some form in this current introduction part of this paper some results along the main body of this paper below. First motivation comes from directly following \cite{KL16} where the relative $p$-adic Hodge theory is essentially applied in the corresponding context of rigid analytic spaces in the strictly situation. \cite{KL16} proved many very deep results around the corresponding categories of pseudocoherent relative $(\varphi,\Gamma)$-modules, namely \cite{KL16} proved that these are abelian categories although working in Banach categories does cause essential difficulties. In our situation we seek the same goal when we deform the corresponding Frobenius modules in arithmetic family. In the context and notation in \cref{proposition6.3} we have:

\begin{proposition} \mbox{\bf{(After Kedlaya-Liu, \cite[Theorem 8.10.6]{KL16})}} 
Working over period ring $\widetilde{\Pi}_{X,A}$, we have the corresponding category of $(\varphi^a,\Gamma)$-modules is abelian, which is basically compatible with the corresponding category of all the sheaves of modules over $\widetilde{\Pi}_{X,A}$ when one would like to form the corresponding kernels and cokernels. $A$ is assumed to be sousperfectoid as those considered in \cite{KH}.
\end{proposition}

\indent Then following Temkin's notes \cite[Part I Chapter I]{DFN} we consider the corresponding analytic spaces which could admit an atlas consisting of all the $k_\Delta$-affinoids namely those quotients of $k_\Delta$ strictly rational localizations of Tate algebras. Over such space $X$: 

\begin{proposition} \mbox{\bf{(After Kedlaya-Liu, \cite[Theorem 8.10.6]{KL16})}} 
Working over period ring $\widetilde{\Pi}_{X,A}$, we have the corresponding category of $(\varphi^a,\Gamma)$-modules is abelian, which is basically compatible with the corresponding category of all the sheaves of modules over $\widetilde{\Pi}_{X,A}$ when one would like to form the corresponding kernel and cokernel. Here $A$ is $\mathbb{Q}_p$.
\end{proposition}

\begin{proposition} \mbox{\bf{(After Kedlaya-Liu, \cite[Theorem 8.10.6]{KL16})}} 
Working over period ring $\widetilde{\Pi}_{X,A}$, we have the corresponding category of $(\varphi^a,\Gamma)$-modules is abelian, which is basically compatible with the corresponding category of all the sheaves of modules over $\widetilde{\Pi}_{X,A}$ when one would like to form the corresponding kernel and cokernel. Here $X$ is assumed to be smooth. $A$ is assumed to be sousperfectoid as those considered in \cite{KH}.\\
\end{proposition}

\indent Then we consider arithmetic family version of Higgs bundles after \cite{LZ} where a version of Riemann-Hilbert correspondence in the arithmetic setting is achieved. The idea in our situation will be definitely the corresponding version of such correspondence for general $A$-relative $(\varphi,\Gamma)$-modules $M$ in the 'geometric' setting. We also consider the corresponding equivalent version of $B$-pairs in the $A$-relative setting after \cite{KL16}. With the notation in later discussion (see \cref{proposition4.45} and \cref{proposition4.42}) we have for a de Rham $A$-relative $(\varphi,\Gamma)$-module $M$ or equivalently a de Rham $A$-relative $B$-pair $M$:

\begin{proposition} \mbox{\bf{(After Kedlaya-Liu, \cite[below Definition 10.10]{KL3})}}
Over a general rigid analyic space $X$, consider a de Rham $A$-relative $(\varphi,\Gamma)$-modules $M$ or equivalently a de Rham $A$-relative $B$-pair $M$. We then have that the corresponding each higher de Rham derived cohomology group $D^i_\mathrm{dR}(M)$ for each $i\geq 0$ is coherent sheaf over $X$. The projectivity could achieved if we assume that $X$ is smooth. $A$ is assumed to be sousperfectoid as those considered in \cite{KH}.

\end{proposition}

\begin{proposition} \mbox{\bf{(After Kedlaya-Liu, \cite[below Definition 10.10]{KL3})}}
Over a general rigid analyic space $X$, consider a de Rham $A$-relative $(\varphi,\Gamma)$-modules $M$ or equivalently a de Rham $A$-relative $B$-pair $M$. We then have that the corresponding each higher de Rham derived cohomology group $E^i_\mathrm{dR}(M)$ for each $i\geq 0$ is coherent sheaf over $X$. The projectivity could achieved if we assume that $X$ is smooth. $A$ is assumed to be sousperfectoid as those considered in \cite{KH}.\\

\end{proposition}

\indent The other corresponding application we want to search for while again along some idea proposed in \cite{KP} applications in equivariant Iwasawa theory. We choose to consider equivariant version of the Iwasawa theory of de Rham $(\varphi,\Gamma)$-module $M$. 

\begin{remark}
We believe (while motivated by \cite{KP}) our idea should eventually produce more than cyclotomic consideration by deforming Berger's $(\varphi,\nabla,\Gamma)$-modules over pro-\'etale site which could absorb the $\Gamma$-action (which could be reconstructed through the Iwasawa deformation).
\end{remark}

\begin{definition} \mbox{\bf{(After Perrin-Riou and Nakamura, \cite[Definition 3.7]{Nakamura1})}} 
With the notation in \cref{section5.1} especially the corresponding abelian Fr\'echet-Stein space we are working on, we have the following derived big Perrin-Riou-Nakamura Exponential map:
\begin{displaymath}
\mathrm{Exp}^{\Pi^\infty(\Gamma'),\bullet}_{\mathcal{F}(M),m}: R^\bullet\Gamma_\mathrm{sheaf}(\varprojlim\mathcal{F}(\Theta_{\mathrm{Ber,dif}}(M))_p) \rightarrow R^\bullet\Gamma_\mathrm{sheaf}(\varprojlim\mathcal{F}(M)_p),
\end{displaymath}
which is corresponding to the one for $(\varphi,\Gamma)$-modules:
\begin{displaymath}
\mathrm{Exp}^{\Pi^\infty(\Gamma'),\bullet}_{M,m}: R^\bullet\Gamma_{\varphi,\Gamma}(\varprojlim \Theta_{\mathrm{Ber,dif}}(M)_p) \rightarrow R^\bullet\Gamma_{\varphi,\Gamma}(\varprojlim M_p).
\end{displaymath}
\end{definition}

\indent We actually conjectured that the corresponding this equivariant Perrin-Riou-Nakamura map in algebraic way produces the equivariant $p$-adic $L$-functions (in some not necessarily explicit way)  which relate directly to the corresponding characteristic ideals. See \cref{section5.2}.\\

\subsection{Future Study}

There are many interesting possible extensions of the current paper. We would like to study more in the future, although here we mainly emphasize the corresponding Iwasawa theoretic aspects. The first scope of topics we would like to study further is the relative $p$-adic Hodge theory in rigid family and the arithmetic Riemann-Hilbert correspondence in rigid family for $B$-pairs after \cite{LZ}, \cite{TT} and \cite{Shi}, along the corresponding foundation we established here.\\

The second scope of the topics we would like to study in the future is to extend the corresponding discussion to more general analytic spaces. In fact here we do not go beyond \cite[Chapter 8]{KL16}, although we touched the corresponding $k_\Delta$-analytic spaces as those considered in Temkin's lecture notes \cite[Part I Chapter I]{DFN}. Namely we still restrict ourselves to the corresponding varieties covered by strictly or non-strictly affinoids. \\

The Iwasawa consideration we considered here is actually also worthwhile to be amplified and to be noncommutativized. Certainly we want to study noncommutative deformation of the sheaves in the future not only in all kinds of Iwasawa theories, but also we would like to study the noncommutative deformation of the corresponding the relative $p$-adic Hodge theory in rigid family and the arithmetic Riemann-Hilbert correspondence in rigid family for $B$-pairs after \cite{LZ}, \cite{TT} and \cite{Shi} in the sense of many existing noncommutative geometric contexts.

\newpage

\section{Logarithmic Hodge-Iwasawa Structures in Mixed-Characteristic Case}\label{section3}

\subsection{$\Gamma$-Modules over Ramified Towers}

\begin{setting}
Let $E$ be $\mathbb{Q}_p$. We use the corresponding $\pi$ to denote the corresponding uniformizer in a uniform notation. Let $A$ be an affinoid algebra over the corresponding field $E$ in our current situation defined through some strict quotient from the corresponding Tate algebras.
\end{setting}

\indent We first consider the corresponding standard ones for $(P,P^+)$ a perfectoid adic Banach uniform pair over $\mathcal{O}_E$, containing the perfectoid field $\mathbb{Q}_p(\zeta_{p^\infty})^\wedge$.

\begin{setting}
In the mixed characteristic situation, we consider the setting in the following fashion:
\begin{align}
H_0&=P\{x_1/t_1,...,x_k/t_k,t_1/y_1,...,t_\ell/y_\ell\},\\
H_0^+&=P^+\{x_1/t_1,...,x_k/t_k,t_1/y_1,...,t_\ell/y_\ell\},\\
H_n&=P\{(x_1/t_1)^{1/p^n},...,(x_k/t_k)^{1/p^n},(t_1/y_1)^{1/p^n},...,(t_\ell/y_\ell)^{1/p^n}\},\\
n&=0,1,..., 0<k\leq\ell.	
\end{align}		
And here the corresponding quantities $x_i,y_j,i=1,...,k,j=1,...,\ell$ could be allowed to be real numbers.
\end{setting}

\begin{lemma} \mbox{\bf{(Kedlaya-Liu \cite[Lemma 7.3.3]{KL16})}}
The corresponding standard ramified toric towers defined as above are weakly decompleting.	
\end{lemma}

\begin{proof}
See \cite[Lemma 7.3.3]{KL16}.	
\end{proof}

\indent The corresponding tower is not finite \'etale, so the corresponding machinery in \cite[Chapter 5]{KL16} needs to be modified, which is already done in the corresponding context in \cite[Chapter 7]{KL16}. As in \cite[Chapter 7, Section 3, Section 4]{KL16} one can also consider direct semilinear $\Gamma$-action and define the corresponding $\Gamma$-modules to be certain ones carrying such semilinear action.

\indent  When we in some situations have the corresponding topological group $\Gamma$ (just as in the corresponding logarithmic setting we are considering here) we can simply consider the corresponding definition as in \cite[Definition 7.3.5]{KL16}:

\begin{definition} \mbox{\bf{(After Kedlaya-Liu \cite[Definition 7.3.5]{KL16})}}
In our current situation, we consider the corresponding action of the corresponding group $\Gamma$ in our situation, namely the $\ell$-fold product of the corresponding additive group $\mathbb{Z}_p$ on the corresponding period rings taking the form of $*_{H,A}$. We now define the corresponding $\Gamma$-modules in the finite projective, pseudocoherent or fpd setting to the corresponding modules over the corresponding period rings as above carrying the corresponding action coming from the group $\Gamma$ which is assumed to as in \cite[Definition 7.3.5]{KL16} particularly semilinear. 	
\end{definition}

\begin{remark}
Again the deformation only happens over the rings in mixed-characteristic situation.	
\end{remark}

\begin{proposition} \mbox{\bf{(Kedlaya-Liu \cite[Lemma 7.3.6]{KL16})}} \label{proposition4.8}
Let $\Gamma_n:=p^n\Gamma$ for $n\geq 0$, then we have in our situation the corresponding result that the corresponding complex $C^\bullet(\Gamma_n,\overline{\varphi}^{-1}R_H/R_H)$ is strict exact.
\end{proposition}

\begin{proof}
See \cite[Lemma 7.3.6]{KL16}.
\end{proof}

\begin{proposition} \mbox{\bf{(Kedlaya-Liu \cite[Lemma 5.6.4]{KL16})}} \label{proposition4.9}
Let $\Gamma_n:=p^n\Gamma$ for $n\geq 0$, then we have in our situation the corresponding result that the corresponding complex $C^\bullet(\Gamma_n,\overline{R}_H/R_H)$ is strict exact.
\end{proposition}

\begin{proof}
See \cite[Lemma 5.6.4]{KL16}. 
\end{proof}

\begin{lemma} \mbox{\bf{(After Kedlaya-Liu \cite[Corollary 5.6.5]{KL16})}} \label{lemma4.10}
We need to fix some suitable $r_0>0$ as in \cite[Corollary 5.6.5]{KL16} and for the corresponding radii $s,r$ in $(0,r_0]$ we will have the corresponding complexes $C^\bullet(\Gamma_n,\varphi^{-1}\Pi^{\mathrm{int},r/p}_{H}/\Pi^{\mathrm{int},r}_{H})_A$ and $C^\bullet(\Gamma_n,\widetilde{\Pi}^{\mathrm{int},r}_{H}/\Pi^{\mathrm{int},r}_{H})_A$ are strict exact. Moreover we have that the corresponding complexes $C^\bullet(\Gamma_n,\varphi^{-1}\Pi^{[s/p,r/p]}_{H}/\Pi^{[s,r]}_{H})_A$ and $C^\bullet(\Gamma_n,\widetilde{\Pi}^{[s/p,r/p]}_{H}/\Pi^{[s,r]}_{H})_A$ are then in our situation strict exact as well.  	
\end{lemma}

\begin{proof}
See \cite[Corollary 5.6.5]{KL16}.	
\end{proof}

\begin{lemma} \mbox{\bf{(After Kedlaya-Liu \cite[Lemma 5.6.6]{KL16})}} \label{lemma4.11}
Let $M$ be a $\Gamma$-module defined over the period rings involved below. Then one can find a sufficiently large integer $\ell^*\geq 0$ such that we have that for any $\ell\geq \ell^*$, the following complexes are correspondingly strict exact: 
\begin{align}
M\otimes_{R_H} C^\bullet(\Gamma_n,\overline{\varphi}^{-\ell-1}R_H/\overline{\varphi}^{-\ell}R_H)&, M\otimes_{R_H} C^\bullet(\Gamma_n,\overline{R}_H/\overline{\varphi}^{-\ell}R_H),\\
M\otimes_{\Pi^{\mathrm{int},r}_{H,A}} C^\bullet(\Gamma_n,\varphi^{-(\ell+1)}\Pi^{\mathrm{int},r/p^{\ell+1}}_{H}/\varphi^{-\ell}\Pi^{\mathrm{int},r/p^\ell}_{H})_A&, M\otimes_{\Pi^{\mathrm{int},r}_{H}} C^\bullet(\Gamma_n,\widetilde{\Pi}^{\mathrm{int},r}_{H}/\varphi^{-\ell}\Pi^{\mathrm{int},r/p^\ell}_{H})_A\\ 
M\otimes_{\Pi_{H,A}} C^\bullet(\Gamma_n,\varphi^{-(\ell+1)}\Pi^{[s/p^{\ell+1},r/p^{\ell+1}]}_{H}/\varphi^{-\ell}\Pi^{[s/p^{\ell},r/p^{\ell}]}_{H})_A&, M\otimes_{\Pi^{\mathrm{int},r}_{H}} C^\bullet(\Gamma_n,\widetilde{\Pi}_{H}/\varphi^{-\ell}\Pi^{[s/p^{\ell},r/p^{\ell}]}_{H})_A.
\end{align}

\end{lemma}

\begin{proof}
This is the corresponding corollary of \cref{proposition4.8}, \cref{proposition4.9} and \cref{lemma4.10}, along the corresponding argument in \cite[Lemma 5.6.6]{KL16}. 	
\end{proof}

\begin{theorem}\mbox{\bf{(After Kedlaya-Liu \cite[Lemma 5.6.9]{KL16})}} \label{theorem4.12}
The corresponding base change functor from the ring $\breve{R}_H$ to $\widetilde{R}_H$ is in our situation an equivalence between the corresponding $\Gamma$-modules.	 The corresponding base change functor from the ring $\breve{\Pi}^{\mathrm{int}}_{H,A}$ to $\widetilde{\Pi}^\mathrm{int}_{H,A}$ is in our situation an equivalence between the corresponding $\Gamma$-modules. The corresponding base change functor from the ring $\breve{\Pi}_{H,A}$ to $\widetilde{\Pi}_{H,A}$ is in our situation an equivalence between the corresponding $\Gamma$-modules.
\end{theorem}

\begin{proof}
We briefly mention the corresponding argument since this is just analog of the corresponding argument for the proof of the corresponding result \cite[Lemma 5.6.9]{KL16}. The corresponding fully faithfulness comes from the \cref{lemma3.10}. Then we apply the corresponding \cite[Lemma 5.6.8]{KL16} to descend the corresponding modules over the bigger rings with accent $\widetilde{*}$ to the base rings after some Frobenius action. Then the corresponding differentials on the original modules over the bigger rings could be made have norms just a small quantity twist away from the corresponding differentials on the corresponding base changes from the corresponding base rings after Frobenius twists. Then to finish we apply the corresponding \label{lemma3.10} to modify this quantity to finish the proof as in \cite[Lemma 5.6.9]{KL16}.
\end{proof}

\indent For the relative logarithmic toric towers, we consider the corresponding following mixed characteristic situation. We consider the following tower $(H,H^+)$:
\begin{align}
H_0&=\mathbb{Q}_p\{x_1/t_1,...,x_k/t_k,t_1/y_1,...,t_\ell/y_\ell\},\\
H_0^+&=H_0^\circ,\\
H_n&=\mathbb{Q}_p(\zeta_{p^n})\{(x_1/t_1)^{1/p^n},...,(x_k/t_k)^{1/p^n},(t_1/y_1)^{1/p^n},...,(t_k/y_k)^{1/p^n}\}, n=0,1,....	
\end{align}

The corresponding parameters $x_i,y_i$ are in $p^\mathbb{Q}$, and $\ell\geq k\geq 0$. The corresponding tower is again not Galois. Again when we in some situations have the corresponding topological group $\Gamma$ (just as in the corresponding logarithmic setting we are considering here) we can simply consider the corresponding definition as in \cite[Definition 7.3.5]{KL16}:

\begin{definition} \mbox{\bf{(After Kedlaya-Liu \cite[Definition 7.3.5]{KL16})}}
In our current situation, we consider the corresponding action of the corresponding group $\Gamma$ in our situation, namely the semidirect product taking the form of $\mathbb{Z}^\times_p\ltimes \mathbb{Z}_p^\ell$ on the corresponding period rings taking the form of $*_{H,A}$. We now define the corresponding $\Gamma$-modules in the finite projective, pseudocoherent or fpd setting to the corresponding modules over the corresponding period rings as above carrying the corresponding action coming from the group $\Gamma$ which is assumed to as in \cite[Definition 7.3.5]{KL16} particularly semilinear. 	
\end{definition}

\begin{proposition} \mbox{\bf{(After Kedlaya-Liu \cite[Lemma 7.3.6]{KL16})}} 
Let $\Gamma_n:=p^n\Gamma$ for $n\geq 0$, then we have in our situation the corresponding result that the corresponding complex $C^\bullet(\Gamma_n,\overline{\varphi}^{-1}R_H/R_H)$ is strict exact.
\end{proposition}

\begin{proof}
See \cref{proposition4.8}.
\end{proof}

\begin{proposition} \mbox{\bf{(After Kedlaya-Liu \cite[Lemma 5.6.4]{KL16})}} 
Let $\Gamma_n:=p^n\Gamma$ for $n\geq 0$, then we have in our situation the corresponding result that the corresponding complex $C^\bullet(\Gamma_n,\overline{R}_H/R_H)$ is strict exact.
\end{proposition}

\begin{proof}
See \cref{proposition4.9}.
\end{proof}

\begin{lemma} \mbox{\bf{(After Kedlaya-Liu \cite[Corollary 5.6.5]{KL16})}} 
We need to fix some suitable $r_0>0$ as in \cite[Corollary 5.6.5]{KL16} and for the corresponding radii $s,r$ in $(0,r_0]$ we will have the corresponding complexes $C^\bullet(\Gamma_n,\varphi^{-1}\Pi^{\mathrm{int},r/p}_{H}/\Pi^{\mathrm{int},r}_{H})_A$ and $C^\bullet(\Gamma_n,\widetilde{\Pi}^{\mathrm{int},r}_{H}/\Pi^{\mathrm{int},r}_{H})_A$ are strict exact. Moreover we have that the corresponding complexes $C^\bullet(\Gamma_n,\varphi^{-1}\Pi^{[s/p,r/p]}_{H}/\Pi^{[s,r]}_{H})_A$ and $C^\bullet(\Gamma_n,\widetilde{\Pi}^{[s/p,r/p]}_{H}/\Pi^{[s,r]}_{H})_A$ are then in our situation strict exact as well.  	
\end{lemma}

\begin{proof}
See \cref{lemma4.10}.	
\end{proof}

\begin{lemma} \mbox{\bf{(After Kedlaya-Liu \cite[Lemma 5.6.6]{KL16})}} 
Let $M$ be a $\Gamma$-module defined over the period rings involved below. Then one can find a sufficiently large integer $\ell^*\geq 0$ such that we have that for any $\ell\geq \ell^*$, the following complexes are correspondingly strict exact: 
\begin{align}
M\otimes_{R_H} C^\bullet(\Gamma_n,\overline{\varphi}^{-\ell-1}R_H/\overline{\varphi}^{-\ell}R_H)&, M\otimes_{R_H} C^\bullet(\Gamma_n,\overline{R}_H/\overline{\varphi}^{-\ell}R_H),\\
M\otimes_{\Pi^{\mathrm{int},r}_{H,A}} C^\bullet(\Gamma_n,\varphi^{-(\ell+1)}\Pi^{\mathrm{int},r/p^{\ell+1}}_{H}/\varphi^{-\ell}\Pi^{\mathrm{int},r/p^\ell}_{H})_A&, M\otimes_{\Pi^{\mathrm{int},r}_{H}} C^\bullet(\Gamma_n,\widetilde{\Pi}^{\mathrm{int},r}_{H}/\varphi^{-\ell}\Pi^{\mathrm{int},r/p^\ell}_{H})_A\\ 
M\otimes_{\Pi_{H,A}} C^\bullet(\Gamma_n,\varphi^{-(\ell+1)}\Pi^{[s/p^{\ell+1},r/p^{\ell+1}]}_{H}/\varphi^{-\ell}\Pi^{[s/p^{\ell},r/p^{\ell}]}_{H})_A&, M\otimes_{\Pi^{\mathrm{int},r}_{H}} C^\bullet(\Gamma_n,\widetilde{\Pi}_{H}/\varphi^{-\ell}\Pi^{[s/p^{\ell},r/p^{\ell}]}_{H})_A.
\end{align}

\end{lemma}

\begin{proof}
See \cref{lemma4.11}.	
\end{proof}

\begin{theorem}\mbox{\bf{(After Kedlaya-Liu \cite[Lemma 5.6.9]{KL16})}} 
The corresponding base change functor from the ring $\breve{R}_H$ to $\widetilde{R}_H$ is in our situation an equivalence between the corresponding $\Gamma$-modules.	 The corresponding base change functor from the ring $\breve{\Pi}^{\mathrm{int}}_{H,A}$ to $\widetilde{\Pi}^\mathrm{int}_{H,A}$ is in our situation an equivalence between the corresponding $\Gamma$-modules. The corresponding base change functor from the ring $\breve{\Pi}_{H,A}$ to $\widetilde{\Pi}_{H,A}$ is in our situation an equivalence between the corresponding $\Gamma$-modules.
\end{theorem}

\begin{proof}
See \cref{theorem4.12}.
\end{proof}

\begin{remark} \label{remark2.17}
The corresponding statements for the mixed characteristic version of the relative log towers are established in \cite[Chapter 7.3, Chapter 7.4]{KL16} without touching the corresponding large coefficient $A$ here, therefore for more detailed discussion see \cite[Chapter 7.3, Chapter 7.4]{KL16}. We should mention that the corresponding \cite[Chapter 7.3, Chapter 7.4]{KL16} target at the corresponding finite projective objects, essentially in the generality above we should come across essentially similar difficulties to possible ones in \cite{KL16} when one would like to consider the pseudocoherent sheaves. Things in our current section are basically the complementary discussion to the discussion we made in \cite{XT2} for the unramified towers.\\  	
\end{remark}


\subsection{$(\varphi,\Gamma)$-Modules}

\indent Now we study the corresponding $\Gamma$-modules carrying further Frobenius action, but we will again mainly focus on the corresponding finite projective objects (see \cref{remark2.17}), which is parallel to the corresponding consideration in \cite{KL16}. First recall the following definition:

\begin{setting} \mbox{\bf{(After Kedlaya-Liu \cite[Definition 7.5.3]{KL16})}} \label{setting2.18}
We will consider the corresponding Kummer \'etale topology and the corresponding pro-Kummer \'etale topology considered in \cite[Definition 7.5.3]{KL16}. Let $(H,H^+)$ be either the standard ramified toric tower or the relative ramified toric tower we considered as obave. Here we recall that a covering of the adic space $\mathrm{Spa}(H_0,H_0^+)$ is called \'etale if it is so after considering the corresponding pullback along for some integer $v\geq 0$ the map $T_1\mapsto T_1^{p^{-v}},...,T_\ell\mapsto T_\ell^{p^{-v}}$. And we recall that a covering of the adic space $\mathrm{Spa}(H_0,H_0^+)$ is called finite \'etale if it is so after considering the corresponding pullback along for some integer $v\geq 0$ the map $T_1\mapsto T_1^{p^{-v}},...,T_\ell\mapsto T_\ell^{p^{-v}}$. And we recall that a covering of the adic space $\mathrm{Spa}(H_0,H_0^+)$ is called faithfully finite \'etale if it is so after considering the corresponding pullback along for some integer $v\geq 0$ the map $T_1\mapsto T_1^{p^{-v}},...,T_\ell\mapsto T_\ell^{p^{-v}}$.	
\end{setting}

\begin{definition}\mbox{\bf{(After Kedlaya-Liu \cite[Definition 5.7.2]{KL16})}}
Consider the corresponding framework in the previous subsection. Consider the corresponding period rings in our situation (with the corresponding notations in \cite{XT1} and \cite{XT2}):
\begin{align}
\Omega^{\mathrm{int}}_{H,A},\breve{\Omega}_{H,A}^{\mathrm{int}},\widehat{\Omega}_{H,A}^{\mathrm{int}},\widetilde{\Omega}_{H,A}^{\mathrm{int}},\\
\Pi^{\mathrm{int}}_{H,A},\breve{\Pi}_{H,A}^{\mathrm{int}},\widehat{\Pi}_{H,A}^{\mathrm{int}},\widetilde{\Pi}_{H,A}^{\mathrm{int}},\\
\Pi^{\mathrm{bd}}_{H,A},\breve{\Pi}_{H,A}^{\mathrm{bd}},\widehat{\Pi}_{H,A}^{\mathrm{bd}},\widetilde{\Pi}_{H,A}^{\mathrm{bd}},\\
\Pi_{H,A},\breve{\Pi}_{H,A},\widetilde{\Pi}_{H,A},
\end{align}
we have the corresponding notions of finite projective or pseudocoherent or finite projective dimension $\Gamma$-modules, where we can then add the corresponding semilinear actions as in \cite[Definition 5.7.2]{KL16} to define the corresponding $(\varphi,\Gamma)$-modules satisfying the corresponding condition on the corresponding Frobenius pullbacks. Then furthermore we can define the same objects for those rings having some radius $r>0$ of interval $[s,r]\in (0,\infty)$ as in \cite[Definition 5.7.2]{KL16}, again satisfying the corresponding Frobenius pullback condition taking the general form:
\begin{align}
&\varphi^*\Delta \otimes_{*_{H,A}^{r/p}} *_{{H,A}}^{r/p} \overset{\sim}{\rightarrow} \Delta \otimes *_{{H,A}}^{r/p},\\
&\varphi^*\Delta \otimes_{*_{H,A}^{[s/p,r/p]}} *_{H,A}^{[s,r/p]} \overset{\sim}{\rightarrow} \Delta \otimes_{*_{H,A}^{[s,r]}} *_{H,A}^{[s,r/p]}.	
\end{align}
\end{definition}

%
%

\begin{proposition} \mbox{\bf{(After Kedlaya-Liu \cite[Theorem 5.7.4]{KL16})}} \label{proposition2.20}
Consider the following period rings in the rational setting:
\begin{align}
\Omega_H,\breve{\Omega}_H,\widehat{\Omega}_H,\widetilde{\Omega}_H,\\
\Pi^{\mathrm{bd}}_H,\breve{\Pi}_H^{\mathrm{bd}},\widehat{\Pi}_H^{\mathrm{bd}},\widetilde{\Pi}_H^{\mathrm{bd}},\\
\Pi_H,\breve{\Pi}_H,\widetilde{\Pi}_H.	
\end{align}
Then the corresponding rational finite projective modules over these rings carrying the corresponding $(\varphi,\Gamma)$-structures. For the deformed we look at the corresponding rings in the following:
\begin{align}
\Pi_{H,A},\breve{\Pi}_{H,A}
,\widetilde{\Pi}_{H,A}.	
\end{align}	
Then we have the corresponding parallel results as well.
\end{proposition}

\begin{proof}
See \cite[Theorem 5.7.4]{KL16}.	
\end{proof}

\begin{remark}
One can certainly compare the corresponding $(\varphi,\Gamma)$-modules above with the corresponding Frobenius sheaves over the Kummer pro-\'etale topology in \label{setting2.18} as in \cite{KL16} by adding additional effective conditions on the objects.
\end{remark}


\newpage

\section{Hodge-Iwasawa Modules}

\subsection{The Categories of Hodge-Iwasawa Modules}

\noindent We recall the corresponding Hodge-Iwasawa modules we defined in the previous papers \cite{XT1} and \cite{XT2}. We use our notations in \cite{XT2} for the corresponding tower and the corresponding period rings.

\begin{setting}
We use the notation $(H,H^+)$ to denote the corresponding tower in our situation. Here we assume that $(H,H^+)$ is finite \'etale, noetherian, weakly-decompleting, decompleting, Galois with the Galois topological group $\Gamma$. We also then follow the corresponding setting of \cite[Hypothesis 3.1.1]{KL16}. And fix some $a>0$ positive integer.	
\end{setting}

\begin{setting}
Recall that we have the corresponding period rings in the following, defined in \cite{KL15} and \cite{KL16}:
\begin{displaymath}
\widetilde{\Omega}_H^\mathrm{int},\widetilde{\Omega}_H,\widetilde{\Pi}_H^{\mathrm{int},r},	\widetilde{\Pi}_H^{\mathrm{int}},\widetilde{\Pi}_H^{\mathrm{bd},r},	\widetilde{\Pi}_H^{\mathrm{bd}},\widetilde{\Pi}_H^{r},	\widetilde{\Pi}_H^{I},\widetilde{\Pi}_H^\infty,	\widetilde{\Pi}_H
\end{displaymath}
form the corresponding group of perfect period rings along the tower $(H,H^+)$, here $r>0$ is some radius while $I$ corresponds to some closed interval. Then as in \cite[Definition 5.2.1]{KL16} we have various imperfection of the corresponding rings listed above. We mainly need the corresponding Robba rings in the following fashion:
\begin{displaymath}
\breve{\Pi}_H^{r},	\breve{\Pi}_H^{I},\breve{\Pi}_H^\infty,	\breve{\Pi}_H,	
\end{displaymath}
and 
\begin{displaymath}
{\Pi}_H^{r},	{\Pi}_H^{I}, \Pi_H^\infty,	{\Pi}_H.	
\end{displaymath}
\end{setting}

\begin{setting}
The Hodge-Iwasawa modules we considered and studied are basically over the corresponding deformed version of the corresponding period rings we recalled above over some affinoid algebra $A$ over the base field $E$ as in \cite{XT2}:
\begin{displaymath}
\breve{\Pi}_{H,A}^{r},	\breve{\Pi}_{H,A}^{I},\breve{\Pi}_{H,A}^\infty,	\breve{\Pi}_{H,A},	
\end{displaymath}
and 
\begin{displaymath}
{\Pi}_{H,A}^{r},	{\Pi}_{H,A}^{I}, \Pi_{H,A}^\infty,	{\Pi}_{H,A},	
\end{displaymath}	
and
\begin{displaymath}
\widetilde{\Pi}_{H,A}^{r},	\widetilde{\Pi}_{H,A}^{I},\widetilde{\Pi}_{H,A}^\infty,	\widetilde{\Pi}_{H,A}.
\end{displaymath}

\end{setting}

\indent Recall that we have the following corresponding key categories involved in our study (\cite{XT2}). First we consider the corresponding finite projective objects for any suitable $s,r$ such that $0\leq s\leq r/p^{ah}$ (for any $A$ such that the corresponding rings building the corresponding Fargues-Fontaine curve namely the perfect Robba rings with respect to some intervals are sheafy as adic rings):

\noindent A. Look at the corresponding space $\mathrm{Spa}(H_0,H_0^+)$ which is the corresponding base of the tower, then we have the corresponding category of all the Frobenius sheaves over the pro-\'etale site of this space over the sheaf of ring $\widetilde{\Pi}_{\mathrm{Spa}(H_0,H_0^+)_\text{pro\'et},A}$. Here the Frobenius will be chosen in our situation to be $\varphi^a$;\\
\noindent B. We look at the same space as in the previous item, and we look at the corresponding category of the Frobenius bundles over the corresponding sheaf of ring $\widetilde{\Pi}_{\mathrm{Spa}(H_0,H_0^+)_\text{pro\'et},A}$;\\
\noindent C. We look at the same base space and then the corresponding sheaf of ring $\widetilde{\Pi}^{[s,r]}_{\mathrm{Spa}(H_0,H_0^+)_\text{pro\'et},A}$ for $s\leq r/p^{ah}$, we then have the category of the corresponding finite projective Frobenius modules over the period rings in this setting;\\
\noindent D. We then have the corresponding finite projective modules over the period ring ${\widetilde{\Pi}}_{H,A}$, carrying the corresponding action coming from $(\varphi^a,\Gamma)$;\\
\noindent E. We then have the corresponding finite projective bundles over the period ring ${\widetilde{\Pi}}_{H,A}$, carrying the corresponding action coming from $(\varphi^a,\Gamma)$;\\
\noindent F. We then have the corresponding finite projective modules over the period ring ${\widetilde{\Pi}}^{[s,r]}_{H,A}$, carrying the corresponding action coming from $(\varphi^a,\Gamma)$, $0\leq s\leq r/p^{ah}$;\\
\noindent G. We then have the corresponding finite projective modules over the period ring ${{\Pi}}_{H,A}$, carrying the corresponding action coming from $(\varphi^a,\Gamma)$;\\
\noindent H. We then have the corresponding finite projective bundles over the period ring ${{\Pi}}_{H,A}$, carrying the corresponding action coming from $(\varphi^a,\Gamma)$;\\
\noindent I. We then have the corresponding finite projective modules over the period ring ${{\Pi}}^{[s,r]}_{H,A}$, carrying the corresponding action coming from $(\varphi^a,\Gamma)$, $0\leq s\leq r/p^{ah}$;\\
\noindent J. We then have the corresponding finite projective modules over the period ring ${\breve{\Pi}}_{H,A}$, carrying the corresponding action coming from $(\varphi^a,\Gamma)$;\\
\noindent K. We then have the corresponding finite projective bundles over the period ring ${\breve{\Pi}}_{H,A}$, carrying the corresponding action coming from $(\varphi^a,\Gamma)$;\\
\noindent L. We then have the corresponding finite projective modules over the period ring ${\breve{\Pi}}^{[s,r]}_{H,A}$, carrying the corresponding action coming from $(\varphi^a,\Gamma)$, $0\leq s\leq r/p^{ah}$;\\
\noindent M. We have the corresponding adic version of the corresponding Fargues-Fontaine curve $FF_{\overline{H}'_\infty,A}$. And we can consider the corresponding category of $\Gamma$-equivariant quasi-coherent locally free sheaves over this adic version of Fargues-Fontaine curve, and we assume the corresponding $\Gamma$-equivariant action to be semilinear and continuous over each $\Gamma$-equivariant affinoid subspace.

\begin{remark}
One has to be careful when one would like to define the objects in $A,B,C$ which could be compared to the other objects. One needs to consider the corresponding locally finite projective objects instead of locally finite free objects.	
\end{remark}

\indent Then we have:

\begin{proposition}\mbox{\bf{(See \cite[Proposition 5.44]{XT2}, after Kedlaya-Liu \cite[Theorem 5.7.4]{KL16})}}
The categories described above are equivalent to each other.	
\end{proposition}

\indent Then we consider the corresponding pseudocoherent objects, where we further assume that the corresponding ring $R_H$ is now $F$-(finite projective) (for any $A$ such that the corresponding rings building the corresponding Fargues-Fontaine curve namely the perfect Robba rings with respect to some intervals are sheafy as adic rings).

\noindent A. Look at the corresponding space $\mathrm{Spa}(H_0,H_0^+)$ which is the corresponding base of the tower, then we have the corresponding category of all the Frobenius sheaves over the pro-\'etale site of this space over the sheaf of ring $\widetilde{\Pi}_{\mathrm{Spa}(H_0,H_0^+)_\text{pro\'et},A}$. Here the Frobenius will be chosen in our situation to be $\varphi^a$;\\
\noindent B. We look at the same space as in the previous item, and we look at the corresponding category of the Frobenius bundles over the corresponding sheaf of ring $\widetilde{\Pi}_{\mathrm{Spa}(H_0,H_0^+)_\text{pro\'et},A}$;\\
\noindent C. We look at the same base space and then the corresponding sheaf of ring $\widetilde{\Pi}^{[s,r]}_{\mathrm{Spa}(H_0,H_0^+)_\text{pro\'et},A}$ for $s\leq r/p^{ah}$, we then have the category of the corresponding pseudocoherent Frobenius modules over the period rings in this setting;\\
\noindent D. We then have the corresponding pseudocoherent modules over the period ring ${\widetilde{\Pi}}_{H,A}$, carrying the corresponding action coming from $(\varphi^a,\Gamma)$;\\
\noindent E. We then have the corresponding pseudocoherent bundles over the period ring ${\widetilde{\Pi}}_{H,A}$, carrying the corresponding action coming from $(\varphi^a,\Gamma)$;\\
\noindent F. We then have the corresponding pseudocoherent modules over the period ring ${\widetilde{\Pi}}^{[s,r]}_{H,A}$, carrying the corresponding action coming from $(\varphi^a,\Gamma)$, $0\leq s\leq r/p^{ah}$;\\
\noindent G. We then have the corresponding pseudocoherent modules over the period ring ${{\Pi}}_{H,A}$, carrying the corresponding action coming from $(\varphi^a,\Gamma)$;\\
\noindent H. We then have the corresponding pseudocoherent bundles over the period ring ${{\Pi}}_{H,A}$, carrying the corresponding action coming from $(\varphi^a,\Gamma)$;\\
\noindent I. We then have the corresponding pseudocoherent modules over the period ring ${{\Pi}}^{[s,r]}_{H,A}$, carrying the corresponding action coming from $(\varphi^a,\Gamma)$, $0\leq s\leq r/p^{ah}$;\\
\noindent J. We then have the corresponding pseudocoherent modules over the period ring ${\breve{\Pi}}_{H,A}$, carrying the corresponding action coming from $(\varphi^a,\Gamma)$;\\
\noindent K. We then have the corresponding pseudocoherent bundles over the period ring ${\breve{\Pi}}_{H,A}$, carrying the corresponding action coming from $(\varphi^a,\Gamma)$;\\
\noindent L. We then have the corresponding pseudocoherent modules over the period ring ${\breve{\Pi}}^{[s,r]}_{H,A}$, carrying the corresponding action coming from $(\varphi^a,\Gamma)$, $0\leq s\leq r/p^{ah}$;\\
\noindent M. We have the corresponding adic version of the corresponding Fargues-Fontaine curve $FF_{\overline{H}'_\infty,A}$. And we can consider the corresponding category of pseudocoherent sheaves over this adic version of Fargues-Fontaine curve.

\begin{remark}
One has to be careful when one would like to define the objects in $A,B,C$ which could be compared to the other objects. One needs to consider the corresponding locally pseudocoherent objects instead of locally pseudocoherent objects. Moreover in this current setting locally we have to impose $p$-adic functional analytic conditions on the pseudocoherent objects when one would like to do the localization in different topologies.
\end{remark}

\indent Then we have:

\begin{proposition}\mbox{\bf{(See \cite[Proposition 5.51]{XT2}, after Kedlaya-Liu \cite[Theorem 5.9.4]{KL16})}}
The categories described above are equivalent to each other.	
\end{proposition}

\begin{proposition} \mbox{\bf{(After Kedlaya-Liu \cite[Theorem 5.9.4]{KL16})}}
The categories above are abelian.	
\end{proposition}

\begin{proof}
By our assumption the corresponding ring $\breve{\Pi}^{[s,r]}_{H,A}$ is coherent.	
\end{proof}

\indent At least in the corresponding finite projective situation, we also have some comparison on the corresponding objects in the logarithmic setting:

\begin{proposition} \mbox{\bf{(After Kedlaya-Liu \cite[Lemma 5.4.11]{KL16})}}
Let $(H,H^+)$ be the one of the corresponding logarithmic towers we encountered in the previous section, we then have the following categories are basically equivalent:\\
\noindent A. Look at the corresponding space $\mathrm{Spa}(H_0,H_0^+)$ which is the corresponding base of the tower, then we have the corresponding category of all the Frobenius sheaves over the pro-\'etale site of this space over the sheaf of ring $\widetilde{\Pi}_{\mathrm{Spa}(H_0,H_0^+)_\text{prok\'et},A}$. Here the Frobenius will be chosen in our situation to be $\varphi^a$;\\
\noindent B. We look at the same space as in the previous item, and we look at the corresponding category of the Frobenius bundles over the corresponding sheaf of ring $\widetilde{\Pi}_{\mathrm{Spa}(H_0,H_0^+)_\text{prok\'et},A}$;\\
\noindent C. We look at the same base space and then the corresponding sheaf of ring $\widetilde{\Pi}^{[s,r]}_{\mathrm{Spa}(H_0,H_0^+)_\text{prok\'et},A}$ for $s\leq r/p^{ah}$, we then have the category of the corresponding finite projective Frobenius modules over the period rings in this setting.\\
\end{proposition}

\begin{proof}
This is because the corresponding log perfectoid affinoids form a basis of the corresponding pro-Kummer \'etale site (see the corresponding development in \cite[Chapter 5.3]{DLLZ2}).
\end{proof}

\begin{proposition} \mbox{\bf{(After Kedlaya-Liu \cite[Theorem 5.7.4]{KL16})}} Let $(H,H^+)$ be the one of the corresponding logarithmic towers we encountered in the previous section. Then we have the following categories of Hodge-Iwasawa modules are equivalent:\\
\noindent D. We then have the corresponding finite projective modules over the period ring ${\widetilde{\Pi}}_{H,A}$, carrying the corresponding action coming from $(\varphi^a,\Gamma)$;\\
\noindent E. We then have the corresponding finite projective bundles over the period ring ${\widetilde{\Pi}}_{H,A}$, carrying the corresponding action coming from $(\varphi^a,\Gamma)$;\\
\noindent F. We then have the corresponding finite projective modules over the period ring ${\widetilde{\Pi}}^{[s,r]}_{H,A}$, carrying the corresponding action coming from $(\varphi^a,\Gamma)$, $0\leq s\leq r/p^{ah}$;\\
\noindent G. We then have the corresponding finite projective modules over the period ring ${{\Pi}}_{H,A}$, carrying the corresponding action coming from $(\varphi^a,\Gamma)$;\\
\noindent H. We then have the corresponding finite projective bundles over the period ring ${{\Pi}}_{H,A}$, carrying the corresponding action coming from $(\varphi^a,\Gamma)$;\\
\noindent I. We then have the corresponding finite projective modules over the period ring ${{\Pi}}^{[s,r]}_{H,A}$, carrying the corresponding action coming from $(\varphi^a,\Gamma)$, $0\leq s\leq r/p^{ah}$;\\
\noindent J. We then have the corresponding finite projective modules over the period ring ${\breve{\Pi}}_{H,A}$, carrying the corresponding action coming from $(\varphi^a,\Gamma)$;\\
\noindent K. We then have the corresponding finite projective bundles over the period ring ${\breve{\Pi}}_{H,A}$, carrying the corresponding action coming from $(\varphi^a,\Gamma)$;\\
\noindent L. We then have the corresponding finite projective modules over the period ring ${\breve{\Pi}}^{[s,r]}_{H,A}$, carrying the corresponding action coming from $(\varphi^a,\Gamma)$, $0\leq s\leq r/p^{ah}$.\\
\end{proposition}

\begin{proof}
The vertical comparison along changing the corresponding ring with respect to the accents was established in \cref{proposition2.20}. Then to compare the corresponding objects in the remaining cases see \cite[Theorem 5.7.4]{KL16} and \cite[Proposition 5.44]{XT2}.	
\end{proof}

\indent In Iwaswa theory, the corresponding aspects of Hodge-Iwasawa theory will usually happen over some Fr\'echet-Stein algebras. Here we will use some geometric language which is different from the situation we considered in our previous work. We now first consider the corresponding commutative setting. We use the notation $X$ to denote a general Stein space:
\begin{displaymath}
X:= \bigcup_{n} X_n	
\end{displaymath}
regarded as a general quasi-Stein adic space in the corresponding context of \cite[Chapter 2.6]{KL16}. Then we have the corresponding sheaf will be organized in the following sense:
\begin{align}
\mathcal{O}_X:=\varprojlim_{n} \mathcal{O}_{X_n}.	
\end{align}

Therefore we will correspondingly consider the period rings over the corresponding spaces taking the form of $X$ as above. We assume each $\mathcal{O}_{X_n}$ satisfies the corresponding assumption as that for the ring $A$ as discussed above. 

\begin{definition}
For any ring $R$ (commutative) which could be written as the corresponding inverse system $\varprojlim_n R_n$. We call the projective system of modules over $R$ as a projective system $\{\mathcal{M}_n\}_n$ of all modules $M_n$ over $R_n$ for each $n\geq 0$.	
\end{definition}

\noindent A. Look at the corresponding space $\mathrm{Spa}(H_0,H_0^+)$ which is the corresponding base of the tower, then we have the corresponding category of all the projective systems of the Frobenius sheaves over the pro-\'etale site of this space over the sheaf of ring $\widetilde{\Pi}_{\mathrm{Spa}(H_0,H_0^+)_\text{pro\'et},\varprojlim_n\mathcal{O}_{X_n}}$. Here the Frobenius will be chosen in our situation to be $\varphi^a$;\\
\noindent B. We look at the same space as in the previous item, and we look at the corresponding category of all the projective systems of the Frobenius bundles over the corresponding sheaf of ring $\widetilde{\Pi}_{\mathrm{Spa}(H_0,H_0^+)_\text{pro\'et},\varprojlim_n\mathcal{O}_{X_n}}$;\\
\noindent C. We look at the same base space and then the corresponding sheaf of ring $\widetilde{\Pi}^{[s,r]}_{\mathrm{Spa}(H_0,H_0^+)_\text{pro\'et},\varprojlim_n\mathcal{O}_{X_n}}$ for $s\leq r/p^{ah}$, we then have the category of the projective systems of  corresponding finite projective Frobenius modules over the period rings in this setting;\\
\noindent D. We then have the corresponding projective systems of  finite projective modules over the period ring ${\widetilde{\Pi}}_{H,\varprojlim_n\mathcal{O}_{X_n}}$, carrying the corresponding action coming from $(\varphi^a,\Gamma)$;\\
\noindent E. We then have the corresponding projective systems of finite projective bundles over the period ring ${\widetilde{\Pi}}_{H,\varprojlim_n\mathcal{O}_{X_n}}$, carrying the corresponding action coming from $(\varphi^a,\Gamma)$;\\
\noindent F. We then have the corresponding projective systems of finite projective modules over the period ring ${\widetilde{\Pi}}^{[s,r]}_{H,\varprojlim_n\mathcal{O}_{X_n}}$, carrying the corresponding action coming from $(\varphi^a,\Gamma)$;\\
\noindent G. We then have the corresponding projective systems of finite projective modules over the period ring ${{\Pi}}_{H,\varprojlim_n\mathcal{O}_{X_n}}$, carrying the corresponding action coming from $(\varphi^a,\Gamma)$;\\
\noindent H. We then have the corresponding projective systems of finite projective bundles over the period ring ${{\Pi}}_{H,\varprojlim_n\mathcal{O}_{X_n}}$, carrying the corresponding action coming from $(\varphi^a,\Gamma)$;\\
\noindent I. We then have the corresponding projective systems of finite projective modules over the period ring ${{\Pi}}^{[s,r]}_{H,\varprojlim_n\mathcal{O}_{X_n}}$, carrying the corresponding action coming from $(\varphi^a,\Gamma)$;\\
\noindent J. We then have the corresponding projective systems of finite projective modules over the period ring ${\breve{\Pi}}_{H,\varprojlim_n\mathcal{O}_{X_n}}$, carrying the corresponding action coming from $(\varphi^a,\Gamma)$;\\
\noindent K. We then have the corresponding projective systems of finite projective bundles over the period ring ${\breve{\Pi}}_{H,\varprojlim_n\mathcal{O}_{X_n}}$, carrying the corresponding action coming from $(\varphi^a,\Gamma)$;\\
\noindent L. We then have the corresponding projective systems of finite projective modules over the period ring ${\breve{\Pi}}^{[s,r]}_{H,\varprojlim_n\mathcal{O}_{X_n}}$, carrying the corresponding action coming from $(\varphi^a,\Gamma)$;\\
\noindent M. We have the corresponding adic version of the corresponding Fargues-Fontaine curve $FF_{\overline{H}'_\infty,\varprojlim_n\mathcal{O}_{X_n}}$. And we can consider the corresponding category of the projective systems of $\Gamma$-equivariant quasi-coherent locally free sheaves over this adic version of Fargues-Fontaine curve, and we assume the corresponding $\Gamma$-equivariant action to be semilinear and continuous over each $\Gamma$-equivariant affinoid subspace.\\

\indent Then we have:

\begin{proposition}\mbox{\bf{(See \cite[Theorem 4.11]{XT2}, after Kedlaya-Liu \cite[Theorem 5.7.4]{KL16})}}
The categories described above are equivalent to each other. 	
\end{proposition}

\indent For the pseudocoherent setting we have the following categories:

\noindent A. Look at the corresponding space $\mathrm{Spa}(H_0,H_0^+)$ which is the corresponding base of the tower, then we have the corresponding category of all the projective systems of the pseudocoherent Frobenius sheaves over the pro-\'etale site of this space over the sheaf of ring $\widetilde{\Pi}_{\mathrm{Spa}(H_0,H_0^+)_\text{pro\'et},\varprojlim_n\mathcal{O}_{X_n}}$. Here the Frobenius will be chosen in our situation to be $\varphi^a$;\\
\noindent B. We look at the same space as in the previous item, and we look at the corresponding category of all the projective systems of the pseudocoherent Frobenius bundles over the corresponding sheaf of ring $\widetilde{\Pi}_{\mathrm{Spa}(H_0,H_0^+)_\text{pro\'et},\varprojlim_n\mathcal{O}_{X_n}}$;\\
\noindent C. We look at the same base space and then the corresponding sheaf of ring $\widetilde{\Pi}^{[s,r]}_{\mathrm{Spa}(H_0,H_0^+)_\text{pro\'et},\varprojlim_n\mathcal{O}_{X_n}}$ for $s\leq r/p^{ah}$, we then have the category of the projective systems of corresponding pseudocoherent Frobenius modules over the period rings in this setting;\\
\noindent D. We then have the corresponding projective systems of  pseudocoherent modules over the period ring ${\widetilde{\Pi}}_{H,\varprojlim_n\mathcal{O}_{X_n}}$, carrying the corresponding action coming from $(\varphi^a,\Gamma)$;\\
\noindent E. We then have the corresponding projective systems of pseudocoherent bundles over the period ring ${\widetilde{\Pi}}_{H,\varprojlim_n\mathcal{O}_{X_n}}$, carrying the corresponding action coming from $(\varphi^a,\Gamma)$;\\
\noindent F. We then have the corresponding projective systems of pseudocoherent modules over the period ring ${\widetilde{\Pi}}^{[s,r]}_{H,\varprojlim_n\mathcal{O}_{X_n}}$, carrying the corresponding action coming from $(\varphi^a,\Gamma)$;\\
\noindent G. We then have the corresponding projective systems of pseudocoherent modules over the period ring ${{\Pi}}_{H,\varprojlim_n\mathcal{O}_{X_n}}$, carrying the corresponding action coming from $(\varphi^a,\Gamma)$;\\
\noindent H. We then have the corresponding projective systems of pseudocoherent bundles over the period ring ${{\Pi}}_{H,\varprojlim_n\mathcal{O}_{X_n}}$, carrying the corresponding action coming from $(\varphi^a,\Gamma)$;\\
\noindent I. We then have the corresponding projective systems of pseudocoherent modules over the period ring ${{\Pi}}^{[s,r]}_{H,\varprojlim_n\mathcal{O}_{X_n}}$, carrying the corresponding action coming from $(\varphi^a,\Gamma)$;\\
\noindent J. We then have the corresponding projective systems of pseudocoherent modules over the period ring ${\breve{\Pi}}_{H,\varprojlim_n\mathcal{O}_{X_n}}$, carrying the corresponding action coming from $(\varphi^a,\Gamma)$;\\
\noindent K. We then have the corresponding projective systems of pseudocoherent bundles over the period ring ${\breve{\Pi}}_{H,\varprojlim_n\mathcal{O}_{X_n}}$, carrying the corresponding action coming from $(\varphi^a,\Gamma)$;\\
\noindent L. We then have the corresponding projective systems of pseudocoherent modules over the period ring ${\breve{\Pi}}^{[s,r]}_{H,\varprojlim_n\mathcal{O}_{X_n}}$, carrying the corresponding action coming from $(\varphi^a,\Gamma)$;\\
\noindent M. We have the corresponding adic version of the corresponding Fargues-Fontaine curve $FF_{\overline{H}'_\infty,\varprojlim_n\mathcal{O}_{X_n}}$. And we can consider the corresponding category of the projective systems of $\Gamma$-equivariant pseudocoherent sheaves over this adic version of Fargues-Fontaine curve, and we assume the corresponding $\Gamma$-equivariant action to be semilinear and continuous over each $\Gamma$-equivariant affinoid subspace.\\

\indent Then we have:
\begin{proposition}\mbox{\bf{(See \cite[Theorem 4.11]{XT2}, after Kedlaya-Liu \cite[Theorem 5.9.6]{KL16})}}
The categories described above are equivalent to each other. And they are actually abelian.\\	
\end{proposition}


\subsection{The Cohomologies of Hodge-Iwasawa Modules}

\indent We then define the corresponding Herr style cohomology in our current deformed setting for the corresponding. Note that the corresponding cohomology over the corresponding Fr\'echet-Stein algebras are very important in the study of Iwasawa conjectures and the corresponding $p$-adic Tamagawa number conjectures.

\indent First keep the following setting:

\begin{setting}
We use the notation $(H,H^+)$ to denote the corresponding tower in our situation. Here we assume that $(H,H^+)$ is finite \'etale, noetherian, weakly-decompleting, decompleting, Galois with the Galois topological group $\Gamma$. We also then follow the corresponding setting of \cite[Hypothesis 3.1.1]{KL16}. And fix some $a>0$ positive integer.	
\end{setting}

\noindent A. Look at the corresponding space $\mathrm{Spa}(H_0,H_0^+)$ which is the corresponding base of the tower, then we have the corresponding category of all the Frobenius sheaves over the pro-\'etale site of this space over the sheaf of ring $\widetilde{\Pi}_{\mathrm{Spa}(H_0,H_0^+)_\text{pro\'et},A}$. Here the Frobenius will be chosen in our situation to be $\varphi^a$;\\
\begin{definition} \mbox{\bf{(After Kedlaya-Liu, \cite[Definition 4.4.4]{KL16})}} For the category in the above (and the corresponding setting in pseudocoherent situation) we define the corresponding complex $C^\bullet_{\varphi}(M)$ of any object $\mathcal{M}$ as the corresponding hypercomplex of the following complex for just the Frobenius operator:
\[
\xymatrix@R+0pc@C+0pc{
0\ar[r] \ar[r] \ar[r] &\mathcal{M} \ar[r]^{\varphi^a-1} \ar[r] \ar[r]  & \mathcal{M} \ar[r] \ar[r] \ar[r]&0.
}
\]	
\end{definition}

\noindent B. We look at the same space as in the previous item, and we look at the corresponding category of the Frobenius bundles over the corresponding sheaf of ring $\widetilde{\Pi}_{\mathrm{Spa}(H_0,H_0^+)_\text{pro\'et},A}$;\\

\noindent C. We look at the same base space and then the corresponding sheaf of ring $\widetilde{\Pi}^{[s,r]}_{\mathrm{Spa}(H_0,H_0^+)_\text{pro\'et},A}$ for $s\leq r/p^{ah}$, we then have the category of the corresponding finite projective Frobenius modules over the period rings in this setting;

\begin{definition} \mbox{\bf{(After Kedlaya-Liu, \cite[Definition 4.4.4]{KL16})}} For the three categories in the above (and the corresponding setting in pseudocoherent situation) we define the corresponding complex $C^\bullet_{\varphi}(M)$ of any object $\mathcal{M}$ as the corresponding hypercomplex of the following complex for just the Frobenius operator:
\[
\xymatrix@R+0pc@C+0pc{
0\ar[r] \ar[r] \ar[r] &\mathcal{M} \ar[r] \ar[r] \ar[r]  & \mathcal{M}\otimes_{\widetilde{\Pi}^{[s,r]}_{\mathrm{Spa}(H_0,H_0^+)_\text{pro\'et},A}}\widetilde{\Pi}^{[s,r/p^{ah}]}_{\mathrm{Spa}(H_0,H_0^+)_\text{pro\'et},A} \ar[r] \ar[r] \ar[r]&0.
}
\]

\end{definition}

\noindent D. We then have the corresponding finite projective modules over the period ring ${\widetilde{\Pi}}_{H,A}$, carrying the corresponding action coming from $(\varphi^a,\Gamma)$;\\
\noindent E. We then have the corresponding finite projective bundles over the period ring ${\widetilde{\Pi}}_{H,A}$, carrying the corresponding action coming from $(\varphi^a,\Gamma)$;\\
\noindent F. We then have the corresponding finite projective modules over the period ring ${\widetilde{\Pi}}^{[s,r]}_{H,A}$, carrying the corresponding action coming from $(\varphi^a,\Gamma)$;\\
\noindent G. We then have the corresponding finite projective modules over the period ring ${{\Pi}}_{H,A}$, carrying the corresponding action coming from $(\varphi^a,\Gamma)$;\\
\noindent H. We then have the corresponding finite projective bundles over the period ring ${{\Pi}}_{H,A}$, carrying the corresponding action coming from $(\varphi^a,\Gamma)$;\\
\noindent I. We then have the corresponding finite projective modules over the period ring ${{\Pi}}^{[s,r]}_{H,A}$, carrying the corresponding action coming from $(\varphi^a,\Gamma)$;\\
\noindent J. We then have the corresponding finite projective modules over the period ring ${\breve{\Pi}}_{H,A}$, carrying the corresponding action coming from $(\varphi^a,\Gamma)$;\\
\noindent K. We then have the corresponding finite projective bundles over the period ring ${\breve{\Pi}}_{H,A}$, carrying the corresponding action coming from $(\varphi^a,\Gamma)$;\\
\noindent L. We then have the corresponding finite projective modules over the period ring ${\breve{\Pi}}^{[s,r]}_{H,A}$, carrying the corresponding action coming from $(\varphi^a,\Gamma)$;\\

\begin{definition} \mbox{\bf{(After Kedlaya-Liu, \cite[Definition 5.7.9]{KL16})}} For the categories (except for the corresponding bundles) in the above (and the corresponding setting in pseudocoherent situation) we define the corresponding complex $C^\bullet_{\varphi,\Gamma}(\mathcal{M})$ of any object $\mathcal{M}$ as the corresponding totalization of the following complex:
\[
\xymatrix@R+0pc@C+0pc{
0\ar[r] \ar[r] \ar[r] &C^\bullet_{\Gamma}(\mathcal{M})\ar[r]^{\varphi^a-1} \ar[r] \ar[r]  &C^\bullet_{\Gamma}(\mathcal{M}) \ar[r] \ar[r] \ar[r]&0.
}
\]	
or 
\[
\xymatrix@R+0pc@C+0pc{
0\ar[r] \ar[r] \ar[r] &C^\bullet_{\Gamma}(\mathcal{M})\ar[r]^{\varphi^a-1} \ar[r] \ar[r]  &C^\bullet_{\Gamma}(\mathcal{M})\otimes \Delta \ar[r] \ar[r] \ar[r]&0,
}
\]
respectively for different situations in the above where $\Delta\in \{{\breve{\Pi}}^{[s,r/p^{ah}]}_{H,A},{{\Pi}}^{[s,r/p^{ah}]}_{H,A},{\widetilde{\Pi}}^{[s,r/p^{ah}]}_{H,A}\}$.
\end{definition}

\noindent M. We have the corresponding adic version of the corresponding Fargues-Fontaine curve $FF_{\overline{H}'_\infty,A}$. And we can consider the corresponding category of $\Gamma$-equivariant quasi-coherent locally free sheaves over this adic version of Fargues-Fontaine curve (with the deformation from the algebra $A$), and we assume the corresponding $\Gamma$-equivariant action to be semilinear and continuous over each $\Gamma$-equivariant affinoid subspace.

\begin{definition} \mbox{\bf{(After Kedlaya-Liu, \cite{KL16})}} For the categories in the above we define the corresponding complex $C^\bullet_{\Gamma}(\mathcal{M})$ of any object $\mathcal{M}$ as the corresponding hypercomplex of the complex $C^\bullet_{\Gamma}(\mathcal{M})$.	 In pseudocoherent situation, we have the parallel definition.
\end{definition}

\begin{remark}
Although we will not basically repeat the corresponding construction of the cohomology as above, but all the definitions on the corresponding cohomologies could be translated to the corresponding construction for the corresponding logarithmic setting and the corresponding Fr\'echet-Stein context we considered seriously above.	
\end{remark}

\indent Based on the corresponding comparison theorem above we can actually consider the corresponding comparison between the corresponding cohomologies. The corresponding statements will take the following forms:

\begin{proposition} \mbox{\bf{(After Kedlaya-Liu, \cite[Theorem 5.7.10]{KL16})}}
Suppose we have a corresponding finite projective module $\mathcal{M}$ over $\Pi_{H,A}$ carrying the corresponding action of $(\varphi^a,\Gamma)$, and we use the notation $\breve{\mathcal{M}}$ and $\widetilde{M}$ to be the corresponding base change of $\mathcal{M}$ from the original base ring to the ring being one of $\breve{\Pi}_{H,A}$ and $\widetilde{\Pi}_{H,A}$ respectively. Then we have the following three complexes are mutually quasi-isomorphic:
\begin{align}
C_{\varphi,\Gamma}^\bullet(\mathcal{M}),C_{\varphi,\Gamma}^\bullet(\breve{\mathcal{M}}),C_{\varphi,\Gamma}^\bullet(\widetilde{\mathcal{M}}).
\end{align}

\end{proposition}

\begin{proposition} \mbox{\bf{(After Kedlaya-Liu, \cite[Theorem 5.7.10]{KL16})}}
Suppose we have a corresponding finite projective module $\mathcal{M}$ over $\Pi_{H,\varprojlim_n\mathcal{O}_{X_n}}$ carrying the corresponding action of $(\varphi^a,\Gamma)$, and we use the notation $\breve{\mathcal{M}}$ and $\widetilde{M}$ to be the corresponding base change of $\mathcal{M}$ from the original base ring to the ring being one of $\breve{\Pi}_{H,\varprojlim_n\mathcal{O}_{X_n}}$ and $\widetilde{\Pi}_{H,\varprojlim_n\mathcal{O}_{X_n}}$ respectively. Then we have the following three complexes are mutually quasi-isomorphic:
\begin{align}
C_{\varphi,\Gamma}^\bullet(\mathcal{M}),C_{\varphi,\Gamma}^\bullet(\breve{\mathcal{M}}),C_{\varphi,\Gamma}^\bullet(\widetilde{\mathcal{M}}).
\end{align}

\end{proposition}

\indent As in \cite{KL16} starting from a corresponding finite projective module carrying the corresponding $(\varphi^a,\Gamma)$-structure, one can compare the corresponding cohomologies of modules over the full Robba rings with those of those modules over some Robba rings with respect to some intervals.

\begin{proposition}\mbox{\bf{(After Kedlaya-Liu, \cite[Theorem 5.7.11]{KL16})}}
Descend $\mathcal{M}$ which is a finite projective module over $\Pi_{H,A}$ carrying the action of $(\varphi^a,\Gamma)$ to the corresponding Robba ring $\Pi^{r_0}_{H,A}$ for some specific positive radius $r_0>0$. Then as in the previous proposition we consider the corresponding base change of the module $\mathcal{M}$ to the corresponding Robba rings $\breve{\Pi}^{[s,r]}_{H,A}$ for $0<s\leq r\leq r_0$ with that $s\leq r/p^{ah}$, and to the corresponding Robba rings $\widetilde{\Pi}^{[s,r]}_{H,A}$. We denote these two modules by $\breve{\mathcal{M}}$ and $\widetilde{\mathcal{M}}$. Then for sufficiently large integer $m\geq 0$ suppose we consider the following three complexes:
\[
\xymatrix@R+0pc@C+0pc{
0\ar[r] \ar[r] \ar[r] &C^\bullet_{\Gamma}(\breve{\mathcal{M}}_{[s,r]})\ar[r]^{\varphi^a-1} \ar[r] \ar[r]  &C^\bullet_{\Gamma}(\breve{\mathcal{M}}_{[s,r/p^{ah}]}) \ar[r] \ar[r] \ar[r]&0,
}
\]
\[
\xymatrix@R+0pc@C+0pc{
0\ar[r] \ar[r] \ar[r] &C^\bullet_{\Gamma}(\widetilde{\mathcal{M}}_{[s,r]})\ar[r]^{\varphi^a-1} \ar[r] \ar[r]  &C^\bullet_{\Gamma}(\widetilde{\mathcal{M}}_{[s,r/p^{ah}]}) \ar[r] \ar[r] \ar[r]&0,
}
\]
\[
\xymatrix@R+0pc@C+0pc{
0\ar[r] \ar[r] \ar[r] &\varphi^{-am}C^\bullet_{\Gamma}(\mathcal{M}_{[s/p^{ahm},r/p^{ahm}]})\ar[r]^{\varphi^a-1} \ar[r] \ar[r]  &\varphi^{-am}C^\bullet_{\Gamma}(\mathcal{M}_{[s/p^{ahm},r/p^{ah(m+1)}]}) \ar[r] \ar[r] \ar[r]&0.
}
\]
Then we have that these three complexes and the corresponding complex $C^\bullet_{\varphi,\Gamma}(\mathcal{M})$ are quasi-isomorphic to each other.	
\end{proposition}

\begin{proof}
See \cite[Theorem 5.7.11]{KL16}.	
\end{proof}

\begin{remark}
We want to mention here that actually the corresponding context should work in general for the tower which is finite \'etale but not actually Galois. Although in our previous work \cite{XT2} we only considered the corresponding towers which are required to be Galois with respect to some Galois topological group $\Gamma$.	
\end{remark}

\newpage

\section{Applications to General Analytic Spaces}

\subsection{Contact with Rigid Analytic Spaces}

\indent We now consider the corresponding context of more global geometric objects, over smooth proper rigid analytic spaces. We certainly will consider very general coefficients namely the corresponding affinoid algebras in the same framework. The two considerations happen to carry the same context coming from the corresponding rigid analytic spaces, but more strictly speaking they will interact in certain situations such as when we look at some nonabelian quotient of the profinite fundamental groups.


\indent We basically consider the following context, where $X$ will be a corresponding smooth proper rigid analytic space over some analytic field $K$. Here we assume that the corresponding analytic field $K$ is containing the $p$-adic number field $\mathbb{Q}_p$. The affinoid algebra $A$ over $\mathbb{Q}_p$ in the following discussion is assumed to be sousperfectoid as those considered in \cite{KH}. Then we can consider the following category.

\begin{definition}\mbox{\bf{(After Kedlaya-Liu, \cite[Definition 8.10.1]{KL16})}}
We use the corresponding notation $D_{\mathrm{pseudo},\widetilde{\Pi}^{[s,r]}_{X,A}}$ ($0<s\leq r/p^{a}$) to denote the corresponding category of all the pseudocoherent $\widetilde{\Pi}^{[s,r]}_{X,A}$ sheaves, which is regarded as a subcategory of the corresponding category of all the corresponding sheaves over $\widetilde{\Pi}^{[s,r]}_{X,A}$, in the corresponding pro-\'etale topology.	
\end{definition}

\indent Also we have the following category of all the $(\varphi^a,\Gamma)$ over the space $X$ as in the above:

\begin{definition}\mbox{\bf{(After Kedlaya-Liu, \cite[Definition 8.10.5]{KL16})}}
We use the corresponding notation $D_{\mathrm{pseudo},\varphi^a,\widetilde{\Pi}_{X,A}}$ to denote the corresponding category of all the pseudocoherent $\widetilde{\Pi}_{X,A}$ sheaves carrying the corresponding Frobenius operator $\varphi^a$, which is regarded as a subcategory of the corresponding category of all the corresponding sheaves over $\widetilde{\Pi}_{X,A}$, in the corresponding pro-\'etale topology.	
\end{definition}

\begin{proposition} \mbox{\bf{(After Kedlaya-Liu, \cite[Theorem 8.10.6]{KL16})}} \label{proposition6.3}
The corresponding category $D_{\mathrm{pseudo},\varphi^a,\widetilde{\Pi}_{X,A}}$ is abelian, where the corresponding kernels and cokernels are compatible with the corresponding category of all the sheaves of $\widetilde{\Pi}_{X,A}$-modules over the corresponding pro-\'etale site.	
\end{proposition}

\begin{proof}
Locally the corresponding object in this category will be $\varphi^a$-sheaf in the corresponding category $D_{\mathrm{pseudo},\varphi^a,\widetilde{\Pi}_{X_0,A}}$ where $X_0$ is a smooth affinoid subspace. We then deduce the corresponding $\Gamma$-action will be realized by the corresponding Galois group of the corresponding toric tower. Then by the comparison we established in the above, we can have the chance to directly compare this to the corresponding $(\varphi^a,\Gamma)$-module over the ring $\breve{\Pi}_{H,A}$ for some restricted tower $H$, which implies the corresponding result since this ring is coherent. 	
\end{proof}

\subsection{Contact with $k_\Delta$-Analytic Spaces}

\indent In this section we are going to consider the corresponding $k_\Delta$-analytic space in the sense of \cite[Part I Chapter I]{DFN}. The corresponding $\Delta$ was denoted by $H$ in \cite[Part I Chapter I]{DFN}, namely a commutative multiplication group satisfying the corresponding relation:
\begin{displaymath}
|k^\times|\subset H \subset \mathbb{R}^\times_+.	
\end{displaymath}
Here for $k$ we assume that this is an analytic field of characteristic $0$. Then we are going to use the notation $\Delta$ to denote a corresponding commutative multiplicative group such that we have the following:
\begin{displaymath}
|k^\times|\subset \Delta \subset \mathbb{R}^\times_+.	
\end{displaymath}

\indent Recall from \cite[Part I Chapter I.3]{DFN} we have the corresponding notation of $k_\Delta$-affinoid algebras and $k_\Delta$-affinoid spaces. These are basically the corresponding non-strictly generalization of the corresponding strictly affinoid algebras and the corresponding spectrum. For instance the corresponding $k_\Delta$-affinoid algebras are defined to be those taking the corresponding form coming form the admissible quotient from some algebra taking the form of:
\begin{displaymath}
k\{T_1/x_1,...,T_d/x_d\}	
\end{displaymath}
for $x_i\in \Delta$ and $i=1,...,d$. Then we have the corresponding notion from this of the affinoid spaces from some suitable geometric point of view, we will use the language of adic spaces as in \cite{KL16}. So we define:

\begin{definition}\mbox{\bf{(After Temkin, \cite[Part I Chapter I.3]{DFN})}}
We define the corresponding $k_\Delta$-affinoid spaces as the corresponding adic space associated some $k_\Delta$-affinoid algebra $R$. Note that this is endowed with the structure sheaf $\mathcal{O}_{\mathrm{Spa}(R,R^+)}$. 
\end{definition}

\indent Then as in \cite[Part I Chapter I]{DFN} we can define the corresponding general $k_\Delta$-analytic spaces as in the following:

\begin{definition}\mbox{\bf{(After Temkin, \cite[Part I Chapter I.3]{DFN})}} 
We define a general $k_\Delta$-analytic space to be an adic space $X$ which could be covered by an atlas forming by the corresponding $k_\Delta$-analytic affinoid spaces. 
	
\end{definition}

\indent We would like to study the corresponding properties of the corresponding sheaves over a general $k_\Delta$-analytic space $X$ as in \cite{KL16}. And certainly eventually we contact with the corresponding Hodge-Iwasawa theory. But we would like to further make some further more detailed discussion around the current generalized context.

\begin{remark}
We should mention that we reproduce many key arguments following \cite[Chapter 8]{KL16} in the following presentation of the parallel results to \cite[Chapter 8]{KL16}. However, if one really wants to work with some larger base field (by taking the valuation group even up to $\mathbb{R}_+^\times$), then one should be able to simplified many arguments below.	
\end{remark}

\begin{remark} \label{remark6.6}
The corresponding desingularization result from Temkin (as in the context of \cite[Remark 8.1.2]{KL16}) in our situation namely we can desingularize a corresponding $k_\Delta$-affinoid algebra $A$ by the corresponding excellence of this algebra as in \cite[Part I Chapter I Fact 3.1.2.1]{DFN}. Namely in the category of schemes we can choose a map with smooth domain $Y\rightarrow \mathrm{Spec}A$ which will have the corresponding analytification in the categories of all the $k_\Delta$-affinoid algebras, again with smooth domain.	
\end{remark}

\indent Then we consider the corresponding correspondence with the Ax-Sen-Tate style result in our context by using the corresponding pro-\'etale site. So we use the notation $f_\text{pro\'et}$ to denote the corresponding morphism $f_\text{pro\'et}:X_\text{pro\'et}\rightarrow X$.

\begin{proposition} \mbox{\bf{(After Kedlaya-Liu, \cite[Theorem 8.2.3]{KL16})}}
Suppose we have that the $k_\Delta$-analytic space $X$ is seminormal. Then we have the corresponding isomorphism $\mathcal{O}_X\overset{\sim}{\rightarrow}f_{\text{pro\'et},*}\widehat{\mathcal{O}}_X$. On the other hand this isomorphism will also imply that the corresponding space $X$ is essentially seminormal.	
\end{proposition}

\begin{proof}
We follow the proof of \cite[Theorem 8.2.3]{KL16}, which proves the smooth situation first and then proves the corresponding general cases by considering the corresponding desingularization. If $k$ is perfectoid and the space $X$ is smooth and affinoid, which as in \cite[Theorem 8.2.3]{KL16} could be directly proved by relating to the corresponding restricted toric towers, then the corresponding result follows from \cite[Theorem 8.2.3]{KL16}. Then to remove the corresponding current assumption on the prefectoidness we just consider the corresponding prefectoidization $k_1$ of $k$ and the corresponding uniform completion of $k_2:=k_1\otimes_k k_1$, and we consider the base changes of $X$ (affinoid) to the corresponding $k_1$ and $k_2$ to form the following short exact sequence:
\[
\xymatrix@R+0pc@C+0pc{
0\ar[r] \ar[r] \ar[r] &X    \ar[r] \ar[r] \ar[r]  &X\otimes k_1 \ar[r] \ar[r] \ar[r] & X\otimes k_2.
}
\]
Then as in \cite[Theorem 8.2.3]{KL16} one further reduces to the corresponding situation around the fields. Finally to consider the general space $X$ we consider the corresponding desingularization as in \cref{remark6.6} to consider some smooth space $Y$ which maps under $h$ say to $X$, then by considering the corresponding pullbacks of perfectoid subdomain one can basically up to such desingularization compare to have the desired isomorphism by the smooth case as above. After this one considers the fact that seminormality implies that we have the isomorphism $\mathcal{O}_{X}\overset{\sim}{\rightarrow} h_* \mathcal{O}_{Y}$, which finishes the proof in this case. Then conversely if we have the corresponding isomorphism in the statement of the proposition, then we choose the corresponding seminormalization $X^{\text{seminor}}$ of $X$ which then finishes the corresponding proof as in \cite[Theorem 8.2.3]{KL16}.

\end{proof}

\begin{proposition}\mbox{\bf{(After Kedlaya-Liu, \cite[Corollary 8.2.4]{KL16})}}
When we consider the corresponding vector bundle $V$ over $X$ which is assumed to be seminormal, we will have then the corresponding isomorphism:
\begin{displaymath}
V\overset{\sim}{\longrightarrow}f_{\text{pro\'et},*}f^*_{\text{pro\'et}}V.	
\end{displaymath}

\end{proposition}

\begin{proof}
From the previous proposition.	
\end{proof}

\begin{proposition}\mbox{\bf{(After Kedlaya-Liu, \cite[Lemma 8.2.7]{KL16})}}
Let $X$ be seminormal, and consider any coherent subsheaf of any locally finite free $\mathcal{O}_X$-module $V$, which is denoted by $W$. We then have the following identity:
\begin{displaymath}
W=f_{\text{pro\'et},*}\mathrm{Image}(f^*_{\text{pro\'et}}W\rightarrow f^*_{\text{pro\'et}}V).	
\end{displaymath}
\end{proposition}

\begin{proof}
As in \cite[Lemma 8.2.7]{KL16}, we basically realize the corresponding coherent subsheaf as the kernel of a map taking the form of $V\rightarrow V'$. Then we have that:
\[
\xymatrix@R+0pc@C+0pc{
0\ar[r] \ar[r] \ar[r] &\mathrm{Image}(f^*_{\text{pro\'et}}W\rightarrow f^*_{\text{pro\'et}}V)    \ar[r] \ar[r] \ar[r]  &f^*_{\text{pro\'et}}V \ar[r] \ar[r] \ar[r] & f^*_{\text{pro\'et}}V'.
}
\]
is exact sequence. Then the previous proposition directly show the corresponding identity after taking the corresponding pushforward along the canonical morphism for the corresponding pro-\'etale site of $X$.
\end{proof}

\begin{proposition}\mbox{\bf{(After Kedlaya-Liu, \cite[Lemma 8.2.8]{KL16})}} Now we consider the following correspondence. Consider a corresponding locally finite free $\mathcal{O}_X$-module $V$ over $X$ which is assumed to be seminormal, and also consider the corresponding locally finite free $\mathcal{O}_X$-sheaf $f^*_{\text{pro\'et}}V$. Then the corresponding functor $W\rightarrow \mathrm{Image}(f^*_{\text{pro\'et}}W\rightarrow f^*_{\text{pro\'et}}V)$ will establish an equivalence between the coherent subsheaves of $V$ and the coherent subsheaves of its pullbacks.
	
\end{proposition}

\begin{proof}
One can finish the proof as in the proof of \cite[Lemma 8.2.8]{KL16}. To recall the argument, we briefly present the proof for the readers. The corresponding fully faithfulness comes from the previous proposition. For the essential surjectivity as in \cite[Lemma 8.2.8]{KL16} we first consider the corresponding smooth case, which could be then reduced to the \cite[Lemma 8.2.6]{KL16} by relating to the corresponding restricted toric tower. In general the argument is by considering any decreasing sequence of ideal of $\mathcal{O}_X$ whose pullback will basically contain some given ideal $B$ in $\widehat{\mathcal{O}}_X$. We denote this sequence by $\{A_i\}_i$. Then for each $A_i$ choose some finite free $C_i$ which maps to $A_i$, and one can basically consider the corresponding inverse image of $B$ in the pullback of $C_i$ which could be denoted by $B_i$, which corresponds to by the previous smooth case a coherent subsheaf of the $C_i$. Now considering further pushforwarding this into $\widehat{\mathcal{O}}_X$ we will have $B_{i+1}$. Now we have if $A_{i+1}=A_i$ for some $i\geq 1$ then we have that $B=B_i$. However this will eventually be the case by the corresponding noetherian property. 	
\end{proof}

\begin{proposition}\mbox{\bf{(After Kedlaya-Liu, \cite[Proposition 8.2.9, Corollary 8.2.10, Corollary 8.2.11]{KL16})}}
A. When we have that $X$ is seminormal, we will have that the corresponding functor $f^*_{\text{pro\'et}}$ establishes embedding from the category of coherent sheaves over $\mathcal{O}_X$ to the category of pseudocoherent sheaves over $\widehat{\mathcal{O}}_X$, which is fully faithful and exact. B. And for any coherent sheaf $F$ over any analytic $k_\Delta$-space, then the tensoring with $F$ will then establish a functor on the category of all the pseudocoherent sheaves over $\widehat{\mathcal{O}}_X$ which is an endofunctor, here $X$ could be allowed to be more general by considering the corresponding seminormalization. C. Over $X$ which is again assumed to be seminormal, then we have the injective map $f^*_{\text{pro\'et}}f_{\text{pro\'et},*}V\hookrightarrow V$ from any locally finite free sheaf $V$ over $\widehat{O}_X$.  	
\end{proposition}

\begin{proof}
See \cite[Proposition 8.2.9, Corollary 8.2.10, Corollary 8.2.11]{KL16}.	
\end{proof}

\indent Then building on the basic facts from \cite[Lemma 8.2.12, Corollary 8.2.15]{KL16} on the corresponding restricted toric towers, we can now study the corresponding categories of the corresponding $\mathcal{O}_X$-sheaves.

\begin{setting}
We now assume we are working over $X=\mathrm{Spa}(L,L^+)$ which is seminormal and affinoid. And consider $(H_\bullet,H_\bullet^+)$ which is assumed to be a perfectoid and finite \'etale tower over the base $X$. Take any ideal $l\subset L$ we can form the corresponding quotient $\overline{L}:=L/JL$ and $\overline{\Omega}:=\overline{H}_\infty/J\overline{H}_\infty$.
\end{setting}

\indent We first have the following:

\begin{proposition}\mbox{\bf{(After Kedlaya-Liu, \cite[Lemma 8.3.3]{KL16})}}
The map $L\rightarrow \overline{H}_\infty$ is faithfully flat. 
\end{proposition}

\begin{proof}
See \cite[Lemma 8.3.3]{KL16}.	
\end{proof}

\begin{proposition}\mbox{\bf{(After Kedlaya-Liu, \cite[Lemma 8.3.4]{KL16})}}
$\Gamma$-modules over the quotient $\overline{\Omega}$ satisfy the following property. The finitely generatedness of any $M$ of such corresponding modules as such could be promoted to be finite projective one if one works over $\overline{\Omega}_e$ with $M_e$, where $e$ exists as an element in $L$ which is not mapped to zero divisors in the quotients $\overline{L}$ and $\overline{\Omega}$.   	
\end{proposition}

\begin{proof}
The corresponding proof is basically parallel to \cite{KL16}. To achieve so we consider the corresponding construction as in the following. First we consider the corresponding minimal ideals of $L$ containing the prescribed ideal $l$. Then for each of such ideal $\mathfrak{p}$ (finitely many) we look at the Fitting ideal $\mathrm{Fitt}_k$ no contained in this prime (for minimized such $k$), choose any element in $\mathrm{Fitt}_k-\mathfrak{p}$, then this will do the job that the corresponding base change of $M$ to $(L/l)_{e_\mathfrak{p}}$ will be finite projective the reasonable quotient. Then to integrate all the information along all such minimal primes we consider for each pair $\mathfrak{p}\neq \mathfrak{q}$ and choose any element $e_{\mathfrak{p}\neq \mathfrak{q}}$. Finally just put $e$ to be $\sum_{\mathfrak{p}}e_\mathfrak{p}\prod_{\mathfrak{p}\neq \mathfrak{q}}e_{\mathfrak{p}\neq \mathfrak{q}}$.	
\end{proof}

\indent The following will be crucial in the understanding of the corresponding category of all the $\widehat{\mathcal{O}}_X$-sheaves over some $k_\Delta$-analytic space $X$.

\begin{proposition}\mbox{\bf{(After Kedlaya-Liu, \cite[Proposition 8.3.5]{KL16})}} \label{proposition615}
Suppose over the ring $\overline{H}_\infty$, we have a finite module which is carrying additionally the structure of $\Gamma$-module. Denote this by $M$. Then this is automatically pseudocoherent and more importantly \'etale-stably pseudocoherent.  
\end{proposition}

\begin{proof}
This is a nontrivial result around the $\Gamma$-modules. We recall the argument in \cite[Proposition 8.3.5]{KL16} in the following, which will prove our result in our current generalized context. First the ideal is using the corresponding noetherian reduction along some prescribed ideal $l\subset L$. As in \cite[Proposition 8.3.5]{KL16} we reduce ourselves to the situation where $\overline{L}$ is reduced and connected. To perform the argument, one first chooses some element $e\in L$ as in the previous proposition such that the statement of the previous proposition hold. Then consider all the corresponding $e$-torsions over $L$ which is denoted by $C$. After applying the corresponding induction hypothesis we have for $M/C$ that the corresponding $M/C/eM/C$ is then pseudocoherent (not known to be strictly or so on). And consider the fact that $(M/C)_e=M_e$ and $M/C[e]$ is trivial, we have then as in \cite[Proposition 8.3.5, Lemma 1.5.8]{KL16}	that $M_C$ is then pseudocoherent (again not known to be strictly or so on). But then this implies the torsion part is then finitely generated, which is killed by some power of $e$ up to some power $h$ for instance. Then repeat considering using the corresponding induction hypothesis we can show that the whole module $M$ is then pseudocoherent (again not known to be strictly or so on) by considering the quotient $e^{x-1}C/e^{x}C$ for $x=0,...,h$. Then to consider the corresponding strictly pseudocoherence, we play with then the kernel along some finite free cover of $M$, which is assumed to be linear over the ring $L$. Then consider kernel and complete the kernel, then apply the corresponding previous paragraph to show that the quotient of $M$ by this kernel is pseudocoherent, which by \cite[Corollary 1.2.11]{KL16} shows that this kernel is actually complete initially, which shows that the corresponding strictly pseudocoherent (not known to be \'etale-stably pseudocoherent). To finish relate the \'etale map out of the period ring to some step in the tower, which shows by the previous argument above we are done.
\end{proof}

\indent We now choose to consider the corresponding parallel discussion as above but for the Robba ring taking the form of $\widetilde{\Pi}_{H}$ and the corresponding related Robba rings. The order of our discussion is different from \cite{KL16}, where we will study the corresponding categories of pseudocoherent sheaves over $\widehat{\mathcal{O}}_X$ and $\widetilde{\Pi}_{X}$  together after we finish more concrete local study. Consider at this moment $X$ which is affinoid defined over a perfectoid field $k$, and consider $(H_\bullet,H_\bullet^+)$ a finite \'etale perfectoid tower as in \cite{KL16}.

\begin{proposition} \mbox{\bf{(After Kedlaya-Liu, \cite[Lemma 8.7.3]{KL16})}}
Consider the ring $\widetilde{\Pi}^{[s,r]}_{H}$ and consider any $\Gamma$-stable nonzero ideal of this ring. Then we have that this is basically containing some nonzero element in $\widetilde{\Pi}_{H}$.	
\end{proposition}

\begin{proof}
Since over a corresponding restricted toric tower this is just the same as in the proof of \cite[Lemma 8.7.3]{KL16} we will not repeat it. And then one can follow the strategy to tackle more general situation by looking at the corresponding $\widetilde{\Pi}^{[s,r]}_X$-sheaf of module associated to the quotient of $\widetilde{\Pi}^{[s,r]}_H$ by the ideal given. Then by applying the previous case one can finish the proof.
\end{proof}

\begin{proposition} \mbox{\bf{(After Kedlaya-Liu, \cite[Corollary 8.7.4]{KL16})}}
Consider the ring $\widetilde{\Pi}^{[s,r]}_{H}$ and consider any $\Gamma$-module which is now assumed to be finite generated over this ring. Then we can find an element such that inverting this element over the base ring the module will be finite projective.	
\end{proposition}

\begin{proof}
See \cite[Corollary 8.7.4]{KL16}.
\end{proof}

\begin{proposition} \mbox{\bf{(After Kedlaya-Liu, \cite[Proposition 8.8.9]{KL16})}}
Under the same condition as in \cite[Proposition 8.8.9]{KL16}, consider the ring $\widetilde{\Pi}^{[s,r]}_{H}$ and consider an element $t$ which is coprime to all the images of $t_\theta$ (see \cite[Proposition 8.8.9]{KL16} for the definition of this element) under all the powers of the Frobenius. By modulo this element any nontrivial ideal of this ring, suppose we have that this ideal will be same under the two corresponding base change to $\widetilde{\Pi}^{[s,r]}_{H^1}$. Then we have that this is just a trivial ideal.
\end{proposition}

\begin{proof}
See \cite[Proposition 8.8.9]{KL16}.
\end{proof}

\begin{proposition} \mbox{\bf{(After Kedlaya-Liu, \cite[Corollary 8.8.10]{KL16})}}
Under the same condition as in \cite[Corollary 8.8.10]{KL16}, consider the ring $\widetilde{\Pi}^{[s,r]}_{H}$ and consider an element $t$ which is coprime to all the images of $t_\theta$ (see \cite[Corollary 8.8.10]{KL16} for the definition of this element) under all the powers of the Frobenius. Then by considering taking the corresponding quotient by this element we will have the chance to promote finitely generatedness to finite projectivity. Now to be more precise suppose over the quotient over $\widetilde{\Pi}^{[s,r]}_{H}$ by this element we have a module which is contained in a finitely generated one and which carries an action from $\Gamma$, then we have that actually this is projective.
\end{proposition}

\begin{proof}
See \cite[Corollary 8.8.10]{KL16}.
\end{proof}

\begin{proposition}\mbox{\bf{(After Kedlaya-Liu, \cite[Proposition 8.9.2]{KL16})}} \label{proposition6.20}
Suppose over the ring $\widetilde{\Pi}_H^{[s,r]}$ for some $0<s\leq r$, we have a finite module which is carrying additionally the structure of $\Gamma$-module. Denote this by $M$. Then this is automatically pseudocoherent and more importantly \'etale-stably pseudocoherent.  
\end{proposition}

\begin{proof}
See \cite[Proposition 8.9.2]{KL16}. This is actually just parallel to \cref{proposition615} by using decomposition of $t$ to supercede the corresponding role of the element $e$.	
\end{proof}

\indent After these technical discussion we can now consider the corresponding categories of the corresponding pseudocoherent sheaves over the ring $\widehat{\mathcal{O}}_X$ and $\widetilde{\Pi}_X^{[s,r]}$ for $0<s\leq r$. Here now $X$ will be then a general $k_\Delta$-analytic space. The affinoid algebra $A$ over $\mathbb{Q}_p$ in the following discussion is assumed to be sousperfectoid as those considered in \cite{KH}.

\begin{definition}\mbox{\bf{(After Kedlaya-Liu, \cite[Definition 8.4.1]{KL16})}}
We use the corresponding notation $D_{\mathrm{pseudo},\widehat{\mathcal{O}}_X}$ to denote the corresponding category of all the pseudocoherent $\widehat{\mathcal{O}}_X$ sheaves over $X$ (certainly under the corresponding pro-\'etale site), which is regarded as a subcategory of the corresponding category of all the corresponding sheaves over $\widehat{\mathcal{O}}_X$, in the corresponding pro-\'etale topology.	
\end{definition}

\begin{proposition} \mbox{\bf{(After Kedlaya-Liu, \cite[Theorem 8.4.3]{KL16})}}
The corresponding category $D_{\mathrm{pseudo},\widehat{\mathcal{O}}_X}$ is abelian, where the corresponding kernels and cokernels are compatible with the corresponding category of all the sheaves of $\widehat{\mathcal{O}}_X$-modules over the corresponding pro-\'etale site.	
\end{proposition}

\begin{proof}
This is a direct consequence of the corresponding properties in \cref{proposition615}.	
\end{proof}

\begin{definition}\mbox{\bf{(After Kedlaya-Liu, \cite[Definition 8.10.1]{KL16})}}
We use the corresponding notation $D_{\mathrm{pseudo},\widetilde{\Pi}^{[s,r]}_{X,A}}$ ($0<s\leq r/p^{a}$) to denote the corresponding category of all the pseudocoherent $\widetilde{\Pi}^{[s,r]}_{X,A}$ sheaves, which is regarded as a subcategory of the corresponding category of all the corresponding sheaves over $\widetilde{\Pi}^{[s,r]}_{X,A}$, in the corresponding pro-\'etale topology.	
\end{definition}

\indent Also we have the following category of all the $(\varphi^a,\Gamma)$ over the space $X$ as in the above:

\begin{definition}\mbox{\bf{(After Kedlaya-Liu, \cite[Definition 8.10.5]{KL16})}}
We use the corresponding notation $D_{\mathrm{pseudo},\varphi^a,\widetilde{\Pi}_{X,A}}$ to denote the corresponding category of all the pseudocoherent $\widetilde{\Pi}_{X,A}$ sheaves carrying the corresponding Frobenius operator $\varphi^a$, which is regarded as a subcategory of the corresponding category of all the corresponding sheaves over $\widetilde{\Pi}_{X,A}$, in the corresponding pro-\'etale topology.	
\end{definition}

\begin{proposition} \mbox{\bf{(After Kedlaya-Liu, \cite[Theorem 8.10.2]{KL16})}}
The corresponding category $D_{\mathrm{pseudo},\widetilde{\Pi}^{[s,r]}_{X,A}}$ is abelian, where the corresponding kernels and cokernels are compatible with the corresponding category of all the sheaves of $\widetilde{\Pi}^{[s,r]}_{X,A}$-modules over the corresponding pro-\'etale site. Here $A$ is just $\mathbb{Q}_p$.	
\end{proposition}

\begin{proof}
This is a direct consequence of the corresponding properties in \cref{proposition6.20}.	
\end{proof}

\begin{proposition} \mbox{\bf{(After Kedlaya-Liu, \cite[Theorem 8.10.6]{KL16})}}
The corresponding category $D_{\mathrm{pseudo},\varphi^a,\widetilde{\Pi}_{X,A}}$ is abelian, where the corresponding kernels and cokernels are compatible with the corresponding category of all the sheaves of $\widetilde{\Pi}_{X,A}$-modules over the corresponding pro-\'etale site. Here $A$ is just $\mathbb{Q}_p$.	
\end{proposition}

\begin{proof}
See the previous proposition.	
\end{proof}

\begin{proposition} \mbox{\bf{(After Kedlaya-Liu, \cite[Theorem 8.10.6]{KL16})}}
When $X$ is smooth. The corresponding category $D_{\mathrm{pseudo},\varphi^a,\widetilde{\Pi}_{X,A}}$ is abelian, where the corresponding kernels and cokernels are compatible with the corresponding category of all the sheaves of $\widetilde{\Pi}_{X,A}$-modules over the corresponding pro-\'etale site. 	
\end{proposition}

\begin{proof}
See \cref{proposition6.3}.	
\end{proof}

\subsection{Contact with Arithmetic Riemann-Hilbert Correspondence}

We now consider the corresponding setting where the base field is finite extension of $\mathbb{F}_p((t))$ or $\mathbb{Q}_p$. The affinoid algebra $A$ over $\mathbb{F}_p((t))$ or $\mathbb{Q}_p$ in the following discussion is assumed to be sousperfectoid as those considered in \cite{KH}. And we consider the corresponding setting above on the corresponding Robba rings. \\

We would like to take this chance to develop a little bit on the corresponding Hodge-Iwasawa $B$-pairs in our context. This was considered by \cite{KP} with some restriction in the representation theoretic context. Since we would like to discuss the corresponding some tower where the internal structure is a little bit complicated to handle (although this is definitely doable). Note that one might want also to have the chance to consider more general tower contacting with the corresponding Iwasawa theory such as the Lubin-Tate as proposed in \cite{KP}.


\begin{definition} \mbox{\bf{(After Kedlaya-Pottharst \cite[Definition 2.17]{KP})}}
Now we first consider more general rigid analytic spaces. Let $X$ be a rigid analytic space over a finite extension $E$ of $\mathbb{Q}_p$ or $\mathbb{F}_p((t))$. Now we consider the corresponding sheaves $\mathbb{B}_{\mathrm{dR},X,A}$, $\mathbb{B}_{e,X,A}$ and $\mathbb{B}^+_{\mathrm{dR},X,A}$ which are defined in the following way. For each of the corresponding perfectoid subdomain $Y$ of the space $X$ in our situation the corresponding rings $\mathbb{B}_{\mathrm{dR},X,A}(Y)$, $\mathbb{B}_{e,X,A}(Y)$ and $\mathbb{B}^+_{\mathrm{dR},X,A}(Y)$ are defined as in the following. So the corresponding map $Y_A \rightarrow \mathrm{Proj} P_{X,A}(Y)$ will by considering the corresponding zero locus define a closed subscheme in the corresponding deformed schematic Fargues-Fontaine curve by $A$. Here in more detail this is just the corresponding map $Y_A \rightarrow \mathrm{Proj} P_{Y^\flat,A}$.
Then we consider the corresponding open complement subspace by removing this subscheme, whose coordinate ring will be the ring $\mathbb{B}_{e,X,A}(Y)$, then we take the corresponding completion of the corresponding deformed schematic Fargues-Fontaine curve along the corresponding closed subscheme to define the corresponding space whose coordinate ring is $\mathbb{B}^+_{\mathrm{dR},X,A}(Y)$, finally taking the corresponding space corresponding to the fiber product of the open affine subscheme and this completion we can take the corresponding coordinate ring which is denoted by $\mathbb{B}_{\mathrm{dR},X,A}(Y)$. Then we take the corresponding graded rings we can define the sheaves $\mathbb{B}^+_{\mathrm{HT},X,A}$ and $\mathbb{B}_{\mathrm{HT},X,A}$ as in \cite[Definition 8.6.5]{KL16}. 
\end{definition}

\begin{definition} \mbox{\bf{(After Tan-Tong, Shimizu \cite[Section 2.1]{TT}, \cite[Section 2]{Shi})}}
In the very parallel way one can consider the corresponding cristalline period sheaves (by taking the corresponding completion in the corresponding de Rham periods above) $\mathbb{B}^+_{\mathrm{cris},X,A}$ and $\mathbb{B}_{\mathrm{cris},X,A}$ as in \cite[Section 2.1]{TT}. In the very parallel way one can consider the corresponding semistable period sheaves $\mathbb{B}^+_{\mathrm{st},X,A}$ and $\mathbb{B}_{\mathrm{st},X,A}$ as in \cite[Section 2]{Shi}.
\end{definition}


Then we have the following generalization of Bloch-Kato \cite{BK1} and Nakamura's \cite{Nakamura1} fundamental sequences in rigid family:

\begin{lemma}  
We then have the corresponding Bloch-Kato fundamental sequence which is defined by the following short exact sequence:
\[
\xymatrix@R+0pc@C+0pc{
0\ar[r] \ar[r] \ar[r] &\mathbb{B}_{e,X,A}\bigcap \mathbb{B}^+_{\mathrm{dR},X,A}    \ar[r] \ar[r] \ar[r]  &\mathbb{B}_{e,X,A}\bigoplus \mathbb{B}^+_{\mathrm{dR},X,A} \ar[r] \ar[r] \ar[r] &\mathbb{B}_{e,X,A}\bigcup \mathbb{B}^+_{\mathrm{dR},X,A} \ar[r] \ar[r] \ar[r] &0 .
}
\]
\end{lemma}

\begin{proof}
This comes from the corresponding construction directly, for instance see \cite[Definition 4.8.2]{KL16}.	
\end{proof}

\begin{definition} \mbox{\bf{(After Scholze, Kedlaya-Liu \cite[Definition 6.8]{Sch1}, \cite[Definition 8.6.5]{KL16})}}
Now we first consider more general rigid analytic spaces.  Let $X$ be a rigid analytic space over a finite extension $E$ of $\mathbb{Q}_p$ or $\mathbb{F}_p((t))$. Now we consider the corresponding sheaves $\mathcal{O}\mathbb{B}_{\mathrm{dR},X,A}$, $\mathcal{O}\mathbb{B}_{e,X,A}$ and $\mathcal{O}\mathbb{B}^+_{\mathrm{dR},X,A}$ which are defined in the following way. $\mathcal{O}\mathbb{B}^+_{\mathrm{dR},X,A}$ is the sheafication of the following presheaf. For any perfectoid subdomain taking the form of $(P_i)_i$ as rings we consider the corresponding product $P^+_i\otimes_{W(\kappa_E)}W(P^{\flat+})$, then take the corresponding $\varpi$-adic completion and invert the corresponding $\varphi$. Then we consider the corresponding completion along the ring $P_i$ and the kernel of the map $\theta$. Then taking the limit throughout all $i$ we can define the corresponding ring $\mathcal{O}\mathbb{B}^+_{\mathrm{dR},X,A}((P_i)_i)$. Then by considering the base change to $\mathbb{B}^+_{\mathrm{dR},X,A}$, one can define the corresponding ring $\mathcal{O}\mathbb{B}^+_{\mathrm{dR},X,A}$. Then by considering the base change to $\mathbb{B}_{\mathrm{dR},X,A}$, one can define the corresponding ring $\mathcal{O}\mathbb{B}_{\mathrm{dR},X,A}$. Then similarly one can define $\mathcal{O}\mathbb{B}_{e,X,A}$. Then taking the corresponding graded pieces we have $\mathcal{O}\mathbb{B}_{\mathrm{HT},X,A}$ and $\mathcal{O}\mathbb{B}^+_{\mathrm{HT},X,A}$ as in \cite[Definition 8.6.5]{KL16}. 
\end{definition}

\begin{definition} \mbox{\bf{(After Tan-Tong, Shimizu \cite[Section 2.1]{TT}, \cite[Section 2]{Shi})}}
In the very parallel way one can consider the corresponding cristalline period sheaves (by taking the corresponding completion in the corresponding de Rham periods above) $\mathcal{O}\mathbb{B}^+_{\mathrm{cris},X,A}$ and $\mathcal{O}\mathbb{B}_{\mathrm{cris},X,A}$ as in \cite[Section 2.1]{TT}. In the very parallel way one can consider the corresponding semistable period sheaves $\mathcal{O}\mathbb{B}^+_{\mathrm{st},X,A}$ and $\mathcal{O}\mathbb{B}_{\mathrm{st},X,A}$ as in \cite[Section 2]{Shi}.
\end{definition}

\indent Then we have the following generalization of Bloch-Kato \cite{BK1} and Nakamura's \cite{Nakamura1} fundamental sequences in rigid family:

\begin{lemma}  
For $X$ smooth, we then have the corresponding Bloch-Kato fundamental sequence which is defined by the following short exact sequence:
\[
\xymatrix@R+0pc@C+0pc{
0\ar[r] \ar[r] \ar[r] &\mathcal{O}\mathbb{B}_{e,X,A}\bigcap \mathcal{O}\mathbb{B}^+_{\mathrm{dR},X,A}    \ar[r] \ar[r] \ar[r]  &\mathcal{O}\mathbb{B}_{e,X,A}\bigoplus \mathcal{O}\mathbb{B}^+_{\mathrm{dR},X,A} \ar[r] \ar[r] \ar[r] &\mathcal{O}\mathbb{B}_{e,X,A}\bigcup \mathcal{O}\mathbb{B}^+_{\mathrm{dR},X,A} \ar[r] \ar[r] \ar[r] &0 .
}
\]
\end{lemma}

\begin{proof}
In the situation where $X$ is smooth, we do have the corresponding explicit expression for the pro-\'etale sheaves, for instance for $\mathcal{O}\mathbb{B}^+_{\mathrm{dR},X,A}$ and $\mathcal{O}\mathbb{B}_{\mathrm{dR},X,A}$ one could follow \cite[Proposition 6.10]{Sch1} to have the chance to find the corresponding expressions which represents these rings as formal power series over the rings $\mathbb{B}^+_{\mathrm{dR},X,A}$ and $\mathbb{B}_{\mathrm{dR},X,A}$. For the ring $\mathcal{O}\mathbb{B}_{e,X,A}$	one could make the parallel argument we will not repeat this again.
\end{proof}

\indent Then we can discuss the corresponding $B$-pairs:

\begin{definition} \mbox{\bf{(After Kedlaya-Liu \cite[Definition 9.3.11]{KL15})}}
	We now define a $B$-pair over the corresponding triplet $(\mathbb{B}_{\mathrm{dR},X,A}, \mathbb{B}_{e,X,A}, \mathbb{B}^+_{\mathrm{dR},X,A})$ is a triplet  $(\mathbb{M}_{\mathrm{dR},X,A}, \mathbb{M}_{e,X,A}, \mathbb{M}^+_{\mathrm{dR},X,A})$ which are the corresponding finite projective objects over the corresponding rings involved above, such that that we have the following isomorphisms:
\begin{displaymath}
\mathbb{M}_{\mathrm{dR},X,A}\overset{\sim}{\longrightarrow}\mathbb{M}_{\mathrm{dR},X,A}^+\otimes_{\mathbb{B}^+_{\mathrm{dR},X,A}}\mathbb{B}_{\mathrm{dR},X,A}\overset{\sim}{\longrightarrow}\mathbb{M}_{e,X,A}^+\otimes_{\mathbb{B}_{e,X,A}} \mathbb{B}_{\mathrm{dR},X,A}.
\end{displaymath}
Similarly as in \cite[Definition 9.4.4]{KL15} we can define the corresponding cohomology of the $B$-pairs by considering the hypercohomology of the following sequence:
\[
\xymatrix@R+0pc@C+0pc{
0\ar[r] \ar[r] \ar[r] &\mathbb{M}_{e,X,A}\bigoplus \mathbb{M}^+_{\mathrm{dR},X,A}    \ar[r] \ar[r] \ar[r]  &\mathbb{M}_{\mathrm{dR},X,A} \ar[r] \ar[r] \ar[r] &0,
}
\]
which is given by subtraction from the first coordinate by the second coordinate in the first non-zero term in this sequence.

\end{definition}

\begin{remark}
The finite projective objects mean those ones which are locally finite projective. Since we are carrying the corresponding coefficient $A$ which is essentially some nontrivial deformation, therefore locally on the space $X$ we have to further glue along $A$ in order to really consider the corresponding finite projective objects which are interesting enough in our development.\\	
\end{remark}

\indent Then we can discuss the corresponding $\mathcal{O}B$-pairs:

\begin{definition} \mbox{\bf{(After Kedlaya-Liu \cite[Definition 9.3.11]{KL15})}}
	We now define a $\mathcal{O}B$-pair over the corresponding triplet $(\mathcal{O}\mathbb{B}_{\mathrm{dR},X,A}, \mathcal{O}\mathbb{B}_{e,X,A}, \mathcal{O}\mathbb{B}^+_{\mathrm{dR},X,A})$ is a triplet $(\mathbb{M}_{\mathrm{dR},X,A}, \mathbb{M}_{e,X,A}, \mathbb{M}^+_{\mathrm{dR},X,A})$ which are the corresponding finite projective objects over the corresponding rings involved above, such that that we have the following isomorphisms:
\begin{displaymath}
\mathbb{M}_{\mathrm{dR},X,A}\overset{\sim}{\longrightarrow}\mathbb{M}_{\mathrm{dR},X,A}^+\otimes_{\mathcal{O}\mathbb{B}^+_{\mathrm{dR},X,A}}\mathcal{O}\mathbb{B}_{\mathrm{dR},X,A}\overset{\sim}{\longrightarrow}\mathbb{M}_{e,X,A}^+\otimes_{\mathcal{O}\mathbb{B}_{e,X,A}} \mathcal{O}\mathbb{B}_{\mathrm{dR},X,A}.
\end{displaymath}
Similarly as in \cite[Definition 9.4.4]{KL15} we can define the corresponding cohomology of the $B$-pairs by considering the hypercohomology of the following sequence:
\[
\xymatrix@R+0pc@C+0pc{
0\ar[r] \ar[r] \ar[r] &\mathbb{M}_{e,X,A}\bigoplus \mathbb{M}^+_{\mathrm{dR},X,A}    \ar[r] \ar[r] \ar[r]  &\mathbb{M}_{\mathrm{dR},X,A} \ar[r] \ar[r] \ar[r] &0,
}
\]
which is given by subtraction from the first coordinate by the second coordinate in the first non-zero term in this sequence.

\end{definition}

We now construct the corresponding deformed geometrization of the corresponding Bloch-Kato exponentials. To do that we consider the following derived functors which are inspired by \cite[Theorem 1.5]{LZ} which could be also dated back to Nakamura's work \cite[2.1]{Nakamura2} in the nonderived situation:

\begin{definition} 
For any $B$-pair we set the corresponding derived de Rham functor $D^\bullet_{\mathrm{dR}}$ and derived Hodge-Tate functor $D^\bullet_{\mathrm{HT}}$ on the $B$-pairs as in the following.
For any $B$-pair we set the corresponding derived cristalline functor $D^\bullet_{\mathrm{cris}}$ and derived semistable functor $D^\bullet_{\mathrm{st}}$ on the $B$-pairs as in the following. Consider a $B$-pair $(\mathbb{M}_{\mathrm{dR},X,A}, \mathbb{M}_{e,X,A}, \mathbb{M}^+_{\mathrm{dR},X,A})$ we have the following functors:
\begin{displaymath}
D_{\mathrm{dR}}^\bullet(\mathbb{M}_{\mathrm{dR},X,A}):=R^\bullet f_{\text{pro\'et},*} \mathbb{M}_{\mathrm{dR},X,A}\otimes_{\mathbb{B}_{\mathrm{dR},X,A}}	 \mathcal{O}\mathbb{B}_{\mathrm{dR},X,A},
\end{displaymath}
and
\begin{displaymath}
D_{\mathrm{HT}}^\bullet(\mathbb{M}_{\mathrm{dR},X,A}):=R^\bullet f_{\text{pro\'et},*} \mathbb{M}_{\mathrm{dR},X,A}\otimes_{\mathbb{B}_{\mathrm{dR},X,A}}	 \mathcal{O}\mathbb{B}_{\mathrm{HT},X,A}
\end{displaymath}
and
\begin{displaymath}
D_{\mathrm{cris}}^\bullet(\mathbb{M}_{\mathrm{dR},X,A}):=R^\bullet f_{\text{pro\'et},*} \mathbb{M}_{\mathrm{dR},X,A}\otimes_{\mathbb{B}_{\mathrm{dR},X,A}}	 \mathcal{O}\mathbb{B}_{\mathrm{cris},X,A}
\end{displaymath}
and
\begin{displaymath}
D_{\mathrm{st}}^\bullet(\mathbb{M}_{\mathrm{dR},X,A}):=R^\bullet f_{\text{pro\'et},*} \mathbb{M}_{\mathrm{dR},X,A}\otimes_{\mathbb{B}_{\mathrm{dR},X,A}}	 \mathcal{O}\mathbb{B}_{\mathrm{st},X,A}.
\end{displaymath}
We call that the $B$-pair de Rham if we have the following isomorphism:
\begin{displaymath}
f^*_\text{pro\'et}(f_{\text{pro\'et},*} \mathbb{M}_{\mathrm{dR},X,A}\otimes_{\mathbb{B}_{\mathrm{dR},X,A}}	 \mathcal{O}\mathbb{B}_{\mathrm{dR},X,A})\otimes\mathcal{O}\mathbb{B}_{\mathrm{dR},X,A} \overset{\sim}{\rightarrow} \mathbb{M}_{\mathrm{dR},X,A}\otimes_{\mathbb{B}_{\mathrm{dR},X,A}}	 \mathcal{O}\mathbb{B}_{\mathrm{dR},X,A},
\end{displaymath}
and we call this Hodge-Tate if we have the corresponding similar isomorphism. We will use the corresponding notation $\mathbb{D}$ to denote the corresponding functor on the derived category level. And we call that the $B$-pair cristalline if we have the corresponding isomorphism:
\begin{displaymath}
f^*_\text{pro\'et}(f_{\text{pro\'et},*} \mathbb{M}_{\mathrm{dR},X,A}\otimes	 \mathcal{O}\mathbb{B}_{\mathrm{cris},X,A})\otimes\mathcal{O}\mathbb{B}_{\mathrm{cris},X,A} \overset{\sim}{\rightarrow} \mathbb{M}_{\mathrm{dR},X,A}\otimes_{\mathbb{B}_{\mathrm{dR},X,A}}	 \mathcal{O}\mathbb{B}_{\mathrm{cris},X,A}.
\end{displaymath}
And we call that the $B$-pair semi-stable if we have the corresponding isomorphism:
\begin{displaymath}
f^*_\text{pro\'et}(f_{\text{pro\'et},*} \mathbb{M}_{\mathrm{dR},X,A}\otimes	 \mathcal{O}\mathbb{B}_{\mathrm{st},X,A})\otimes\mathcal{O}\mathbb{B}_{\mathrm{st},X,A} \overset{\sim}{\rightarrow} \mathbb{M}_{\mathrm{dR},X,A}\otimes_{\mathbb{B}_{\mathrm{dR},X,A}}	 \mathcal{O}\mathbb{B}_{\mathrm{st},X,A}.
\end{displaymath}	
\end{definition}

\indent Now we address some issues around the corresponding finiteness of the construction getting involved as in the above. First when the space $X$ is smooth and proper over the corresponding base fields, in the situation where $A$ is just $\mathbb{Q}_p$ we have in this case $f_\text{pro\'et,*} \mathbb{M}_{\mathrm{dR},X,A}\otimes_{\mathbb{B}_{\mathrm{dR},X,A}}	 \mathcal{O}\mathbb{B}_{\mathrm{dR},X,A}$ of any $B$-pair is coherent sheaf over $\mathcal{O}_X$, which is the consequence of \cite[Theorem 8.6.2 (a)]{KL16}. And moreover in this case since the corresponding space is smooth as noted in \cite[Definition 10.10]{KL3} we have that the corresponding existence of the connection $\nabla$ which will ensure the corresponding coherent sheaves are also projective modules. We now first consider the situation where we have the ring $A$ not just $\mathbb{Q}_p$ in more general situation. Certainly to study the corresponding $B$-pairs it will be basically convenient to relate to the corresponding $(\varphi,\Gamma)$-modules as in the earlier work from Kedlaya-Liu and Nakamura. To do so we can first consider the following higher dimensional generalization of \cite[Theorem 2.18]{KP}:

\begin{remark}
Note that at this moment here we assume that the corresponding base fields are of 0 characteristic.	
\end{remark}

\begin{proposition} \mbox{\bf{(After Kedlaya-Pottharst \cite[Theorem 2.18]{KP})}}
We look at the two groups of objects. In the first consideration we look at the corresponding finite locally free sheaves over the corresponding $A$-deformed version of the schematic Fargues-Fontaine curves attache to any local perfectoid chart for some ring $R$. In the other second consideration we look at the corresponding $A$-deformed version of the corresponding $B$-pairs defined above but after taking the corresponding section over some perfectoid domain associated to some ring $R$. Then we have the corresponding equivalence between the corresponding categories respecting the two groups of objects.	
\end{proposition}

\begin{proof}
As in \cite[Theorem 2.18]{KP} this is by Beauville-Laszlo descent.	
\end{proof}

\begin{remark}
\indent Then we could basically translate the corresponding results on the cohomology of $B$-pairs to those on the corresponding cohomology of the vector bundles over the Fargues-Fontaine curves and further more to those on that of the Frobenius modules over the pro-\'etale site. Then we have that this will be further related to $(\varphi,\Gamma)$-modules over the imperfect Robba rings with respect to the corresponding base spaces. So we move on to consider the corresponding $(\varphi,\Gamma)$-modules over the corresponding pro-\'etale site.	
\end{remark}

\indent Following Nakamura one can further generalize the corresponding Bloch-Kato long exact sequence which provides the corresponding exponentials:

\begin{definition} \mbox{\bf{(After Bloch-Kato-Nakamura)}}
For $X$ smooth, consider any $(\varphi,\Gamma)$-module $M$ over the pro-\'etale sheaf $\widetilde{\Pi}^\infty_{X,A}$ we then take the corresponding pro-\'etale invariance, namely we consider the following two functor with the one $D_\mathrm{dR}$:
\begin{displaymath}
D^\bullet_\mathcal{O}(M):= R^\bullet f_{\text{pro\'et},*}(M\otimes_{\widetilde{\Pi}^\infty_{X,A}} (\mathcal{O}\mathbb{B}_{e,X,A}\bigcap \mathcal{O}\mathbb{B}^+_{\mathrm{dR},X,A})),	
\end{displaymath}
\begin{displaymath}
D^\bullet_\oplus(M):= R^\bullet f_{\text{pro\'et},*}(M\otimes_{\widetilde{\Pi}^\infty_{X,A}} (\mathcal{O}\mathbb{B}_{e,X,A}\bigoplus \mathcal{O}\mathbb{B}^+_{\mathrm{dR},X,A}),	
\end{displaymath}
as well as the corresponding definitions on the derived category level:
\begin{displaymath}
\mathbb{D}^\bullet_\mathcal{O}(M):= R^\bullet f_{\text{pro\'et},*}(M\otimes_{\widetilde{\Pi}^\infty_{X,A}} (\mathcal{O}\mathbb{B}_{e,X,A}\bigcap \mathcal{O}\mathbb{B}^+_{\mathrm{dR},X,A})),	
\end{displaymath}
\begin{displaymath}
\mathbb{D}^\bullet_\oplus(M):= R^\bullet f_{\text{pro\'et},*}(M\otimes_{\widetilde{\Pi}^\infty_{X,A}} (\mathcal{O}\mathbb{B}_{e,X,A}\bigoplus \mathcal{O}\mathbb{B}^+_{\mathrm{dR},X,A}),	
\end{displaymath}

Taking the corresponding complexes induced and considering the corresponding mapping cone we have the following distinguished triangle in $D(\mathcal{O}_X\widehat{\otimes}_{\mathbb{Q}_p}A)$ (the derived category of all the $\mathcal{O}_X\widehat{\otimes}_{\mathbb{Q}_p}A$-modules):
\[
\xymatrix@R+0pc@C+0pc{
C^\bullet_\mathcal{O}(M)\ar[r] \ar[r] \ar[r] &C^\bullet_\oplus(M)    \ar[r] \ar[r] \ar[r]  &C^\bullet_\mathrm{dR}(M) \ar[r] \ar[r] \ar[r] &C^\bullet_\mathcal{O}(M)[1].
}
\]
Consider the corresponding degree $0$ and $1$ we will have geometric and arithmetic family version of the corresponding Bloch-Kato-Nakamura exponential map:
\begin{displaymath}
h^0(C^\bullet_\mathrm{dR}(M))\rightarrow h^1(C^\bullet_\mathcal{O}(M)),	
\end{displaymath}
as the map of sheaves over $\mathcal{O}_X\widehat{\otimes}_{\mathbb{Q}_p}A$. Here we assume that the corresponding $B$-pairs getting involved are de Rham. 

\end{definition}

\begin{definition} \mbox{\bf{(After Kedlaya-Liu \cite[Definition 8.6.5]{KL16})}} 
For any $\varphi$-module $M$ over the sheaf of ring $\widetilde{\Pi}^\infty_{X,A}$ we set the corresponding derived de Rham functor $D^\bullet_{\mathrm{dR}}$ and derived Hodge-Tate functor $D^\bullet_{\mathrm{HT}}$ on the $\varphi$-modules $M$ over the sheaf of ring $\widetilde{\Pi}^\infty_{X,A}$ as in the following.  For any $\varphi$-module $M$ over the sheaf of ring $\widetilde{\Pi}^\infty_{X,A}$ we set the corresponding derived cristalline functor $D^\bullet_{\mathrm{cris}}$ and derived semistable functor $D^\bullet_{\mathrm{st}}$ on the $\varphi$-modules $M$ over the sheaf of ring $\widetilde{\Pi}^\infty_{X,A}$ as in the following. Consider a $\varphi$-module $M$ over the sheaf of ring $\widetilde{\Pi}^\infty_{X,A}$ we have the following functors:
\begin{displaymath}
D_{\mathrm{dR}}^\bullet(M):=R^\bullet f_{\text{pro\'et},*} M\otimes_{\widetilde{\Pi}^\infty_{X,A}}	 \mathcal{O}\mathbb{B}_{\mathrm{dR},X,A},
\end{displaymath}
and
\begin{displaymath}
D_{\mathrm{HT}}^\bullet(M):=R^\bullet f_{\text{pro\'et},*} M \otimes_{\widetilde{\Pi}^\infty_{X,A}}	 \mathcal{O}\mathbb{B}_{\mathrm{HT},X,A}
\end{displaymath}
and
\begin{displaymath}
D_{\mathrm{cris}}^\bullet(M):=R^\bullet f_{\text{pro\'et},*} M\otimes_{\widetilde{\Pi}^\infty_{X,A}}	 \mathcal{O}\mathbb{B}_{\mathrm{cris},X,A},
\end{displaymath}
and
\begin{displaymath}
D_{\mathrm{st}}^\bullet(M):=R^\bullet f_{\text{pro\'et},*} M \otimes_{\widetilde{\Pi}^\infty_{X,A}}	 \mathcal{O}\mathbb{B}_{\mathrm{st},X,A}.
\end{displaymath}
We call that the $B$-pair de Rham if we have the following isomorphism:
\begin{displaymath}
f^*_\text{pro\'et}(f_{\text{pro\'et},*} M\otimes_{\widetilde{\Pi}^\infty_{X,A}}	 \mathcal{O}\mathbb{B}_{\mathrm{dR},X,A})\otimes\mathcal{O}\mathbb{B}_{\mathrm{dR},X,A} \overset{\sim}{\rightarrow} M \otimes_{\widetilde{\Pi}^\infty_{X,A}}	 \mathcal{O}\mathbb{B}_{\mathrm{dR},X,A},
\end{displaymath}
and we call this Hodge-Tate if we have the corresponding similar isomorphism. We will use the corresponding notation $\mathbb{D}$ to denote the corresponding functor on the derived category level. And we define the corresponding properties of being cristalline and semistable in the same parallel way.
	
\end{definition}

\begin{proposition} \label{proposition4.42}
For general rigid analytic space $X$ (note that at this moment we are in the 0 characteristic situation), the corresponding cohomology $D^i_\mathrm{dR}(M)$ is coherent sheaf over $X$ for each $i$. In the situation where the space is smooth we have the further projectivity. 
\end{proposition}

\begin{proof}
This is a relative version of the result of \cite[below Definition 10.10]{KL3}. Namely we consider the corresponding resolution of singularities and reduce to the corresponding smooth case. In the corresponding smooth case we further reduce to the corresponding case for the base of a corresponding toric tower in the restricted sense. Then we relate to pseudocoherent $\Gamma$-modules over $\overline{H}_\infty\widehat{\otimes}A$ in our context, which further relates to the the corresponding pseudocoherent $\Gamma$-modules over $\widetilde{\Pi}_A$ and then over $H_\infty\widehat{\otimes}A$. Then we descend to some finite level ring $H_k\widehat{\otimes}A$ for some $k$. Then we need to consider as in \cite[Corollary 5.9.5]{KL16} that the corresponding p.f. descent will allow us to realize the corresponding $\Gamma$-cohomology as the abstract one by using the higher period rings. Then this will realize the corresponding finiteness. The projectivity follows from the existence of the connection.
\end{proof}

\begin{proposition}
For general rigid analytic space $X$ (note that at this moment we are in the 0 characteristic situation), the corresponding cohomology $D^i_\mathrm{cris}(M),D^i_\mathrm{st}(M)$ are coherent sheaves over $X$ for each $i$. In the situation where the space is smooth we have the further projectivity.	
\end{proposition}

\begin{proof}
See the proof in the previous proposition.
\end{proof}

\begin{corollary}
The corresponding sheaf
\begin{center}
 $f_{\text{pro\'et},*} \mathbb{M}_{\mathrm{dR},X,A}\otimes_{\mathbb{B}_{\mathrm{dR},X,A}}	 \mathcal{O}\mathbb{B}_{\mathrm{dR},X,A}$\end{center}
over $\mathcal{O}_X\widehat{\otimes}A$ is basically a coherent sheaf.	When we have that the corresponding space $X$ is smooth we could further have the corresponding projectivity in our context.  
\end{corollary}

\begin{remark}
The construction above is generalization of \cite[Definition 8.6.5]{KL16}, which is different slightly from \cite[Section 3.2]{LZ}. But we can also push things down to the \'etale site along $g_\text{pro\'et}:X_\text{pro\'et}\rightarrow X_\text{\'et}$.	
\end{remark}

\begin{definition} 
For any $B$-pair we set the corresponding derived de Rham functor $E^\bullet_{\mathrm{dR}}$ and derived Hodge-Tate functor $E^\bullet_{\mathrm{HT}}$ on the $B$-pairs as in the following.
For any $B$-pair we set the corresponding derived cristalline functor $E^\bullet_{\mathrm{cris}}$ and derived semistable functor $E^\bullet_{\mathrm{st}}$ on the $B$-pairs as in the following. Consider a $B$-pair $(\mathbb{M}_{\mathrm{dR},X,A}, \mathbb{M}_{e,X,A}, \mathbb{M}^+_{\mathrm{dR},X,A})$ we have the following functors:
\begin{displaymath}
E_{\mathrm{dR}}^\bullet(\mathbb{M}_{\mathrm{dR},X,A}):=R^\bullet g_{\text{pro\'et},*} \mathbb{M}_{\mathrm{dR},X,A}\otimes_{\mathbb{B}_{\mathrm{dR},X,A}}	 \mathcal{O}\mathbb{B}_{\mathrm{dR},X,A},
\end{displaymath}
and
\begin{displaymath}
E_{\mathrm{HT}}^\bullet(\mathbb{M}_{\mathrm{dR},X,A}):=R^\bullet g_{\text{pro\'et},*} \mathbb{M}_{\mathrm{dR},X,A}\otimes_{\mathbb{B}_{\mathrm{dR},X,A}}	 \mathcal{O}\mathbb{B}_{\mathrm{HT},X,A}
\end{displaymath}
and
\begin{displaymath}
E_{\mathrm{cris}}^\bullet(\mathbb{M}_{\mathrm{dR},X,A}):=R^\bullet g_{\text{pro\'et},*} \mathbb{M}_{\mathrm{dR},X,A}\otimes_{\mathbb{B}_{\mathrm{dR},X,A}}	 \mathcal{O}\mathbb{B}_{\mathrm{cris},X,A}
\end{displaymath}
and
\begin{displaymath}
E_{\mathrm{st}}^\bullet(\mathbb{M}_{\mathrm{dR},X,A}):=R^\bullet g_{\text{pro\'et},*} \mathbb{M}_{\mathrm{dR},X,A}\otimes_{\mathbb{B}_{\mathrm{dR},X,A}}	 \mathcal{O}\mathbb{B}_{\mathrm{st},X,A}.
\end{displaymath}
We call that the $B$-pair de Rham if we have the following isomorphism:
\begin{displaymath}
g^*_\text{pro\'et}(g_{\text{pro\'et},*} \mathbb{M}_{\mathrm{dR},X,A}\otimes_{\mathbb{B}_{\mathrm{dR},X,A}}	 \mathcal{O}\mathbb{B}_{\mathrm{dR},X,A})\otimes\mathcal{O}\mathbb{B}_{\mathrm{dR},X,A} \overset{\sim}{\rightarrow} \mathbb{M}_{\mathrm{dR},X,A}\otimes_{\mathbb{B}_{\mathrm{dR},X,A}}	 \mathcal{O}\mathbb{B}_{\mathrm{dR},X,A},
\end{displaymath}
and we call this Hodge-Tate if we have the corresponding similar isomorphism. We will use the corresponding notation $\mathbb{E}$ to denote the corresponding functor on the derived category level. And we call that the $B$-pair cristalline if we have the corresponding isomorphism:
\begin{displaymath}
g^*_\text{pro\'et}(g_{\text{pro\'et},*} \mathbb{M}_{\mathrm{dR},X,A}\otimes	 \mathcal{O}\mathbb{B}_{\mathrm{cris},X,A})\otimes\mathcal{O}\mathbb{B}_{\mathrm{cris},X,A} \overset{\sim}{\rightarrow} \mathbb{M}_{\mathrm{dR},X,A}\otimes_{\mathbb{B}_{\mathrm{dR},X,A}}	 \mathcal{O}\mathbb{B}_{\mathrm{cris},X,A}.
\end{displaymath}
And we call that the $B$-pair semi-stable if we have the corresponding isomorphism:
\begin{displaymath}
g^*_\text{pro\'et}(g_{\text{pro\'et},*} \mathbb{M}_{\mathrm{dR},X,A}\otimes	 \mathcal{O}\mathbb{B}_{\mathrm{st},X,A})\otimes\mathcal{O}\mathbb{B}_{\mathrm{st},X,A} \overset{\sim}{\rightarrow} \mathbb{M}_{\mathrm{dR},X,A}\otimes_{\mathbb{B}_{\mathrm{dR},X,A}}	 \mathcal{O}\mathbb{B}_{\mathrm{st},X,A}.
\end{displaymath}	
\end{definition}

\begin{definition} \mbox{\bf{(After Kedlaya-Liu \cite[Definition 8.6.5]{KL16})}} 
For any $\varphi$-module $M$ over the sheaf of ring $\widetilde{\Pi}^\infty_{X,A}$ we set the corresponding derived de Rham functor $E^\bullet_{\mathrm{dR}}$ and derived Hodge-Tate functor $E^\bullet_{\mathrm{HT}}$ on the $\varphi$-modules $M$ over the sheaf of ring $\widetilde{\Pi}^\infty_{X,A}$ as in the following.  For any $\varphi$-module $M$ over the sheaf of ring $\widetilde{\Pi}^\infty_{X,A}$ we set the corresponding derived cristalline functor $E^\bullet_{\mathrm{cris}}$ and derived semistable functor $E^\bullet_{\mathrm{st}}$ on the $\varphi$-modules $M$ over the sheaf of ring $\widetilde{\Pi}^\infty_{X,A}$ as in the following. Consider a $\varphi$-module $M$ over the sheaf of ring $\widetilde{\Pi}^\infty_{X,A}$ we have the following functors:
\begin{displaymath}
E_{\mathrm{dR}}^\bullet(M):=R^\bullet g_{\text{pro\'et},*} M\otimes_{\widetilde{\Pi}^\infty_{X,A}}	 \mathcal{O}\mathbb{B}_{\mathrm{dR},X,A},
\end{displaymath}
and
\begin{displaymath}
E_{\mathrm{HT}}^\bullet(M):=R^\bullet g_{\text{pro\'et},*} M \otimes_{\widetilde{\Pi}^\infty_{X,A}}	 \mathcal{O}\mathbb{B}_{\mathrm{HT},X,A}
\end{displaymath}
and
\begin{displaymath}
E_{\mathrm{cris}}^\bullet(M):=R^\bullet g_{\text{pro\'et},*} M\otimes_{\widetilde{\Pi}^\infty_{X,A}}	 \mathcal{O}\mathbb{B}_{\mathrm{cris},X,A},
\end{displaymath}
and
\begin{displaymath}
E_{\mathrm{st}}^\bullet(M):=R^\bullet g_{\text{pro\'et},*} M \otimes_{\widetilde{\Pi}^\infty_{X,A}}	 \mathcal{O}\mathbb{B}_{\mathrm{st},X,A}.
\end{displaymath}
We call that the $B$-pair de Rham if we have the following isomorphism:
\begin{displaymath}
g^*_\text{pro\'et}(g_{\text{pro\'et},*} M\otimes_{\widetilde{\Pi}^\infty_{X,A}}	 \mathcal{O}\mathbb{B}_{\mathrm{dR},X,A})\otimes\mathcal{O}\mathbb{B}_{\mathrm{dR},X,A} \overset{\sim}{\rightarrow} M \otimes_{\widetilde{\Pi}^\infty_{X,A}}	 \mathcal{O}\mathbb{B}_{\mathrm{dR},X,A},
\end{displaymath}
and we call this Hodge-Tate if we have the corresponding similar isomorphism. We will use the corresponding notation $\mathbb{E}$ to denote the corresponding functor on the derived category level. And we define the corresponding properties of being cristalline and semistable in the same parallel way.
	
\end{definition}

\begin{proposition} \label{proposition4.45}
For general rigid analytic space $X$ (note that at this moment we are in the 0 characteristic situation), the corresponding cohomology $E^i_\mathrm{dR}(M)$ for each $i$ is a coherent sheaf over $X$. In the situation where the space is smooth we have the further projectivity.	
\end{proposition}

\begin{proof}
One looks at the following composition:
\begin{align}
X_{\text{pro\'et}}\rightarrow X_{\text{\'et}} \rightarrow X_\mathrm{an}.	
\end{align}
We call the second functor $f_{\text{\'et},\text{an}}$. In effect, the second one realizes the corresponding equivalence between the category of all the coherent $\mathcal{O}_{X_{\text{\'et}}}$-modules and the category of all the coherent $\mathcal{O}_{X_{\text{an}}}$-modules. We have that $f_{\text{pro\'et},*}=f_{\text{\'et},\text{an},*}g_{\text{pro\'et},*}$. Then the result follows from the corresponding \cref{proposition4.42}.
\end{proof}

\begin{proposition} 
For general rigid analytic space $X$ (note that at this moment we are in the 0 characteristic situation), the corresponding cohomology $E^i_\mathrm{cris}(M),E^i_\mathrm{st}(M)$ for each $i$ are coherent sheaves over $X$. In the situation where the space is smooth we have the further projectivity.	
\end{proposition}

\begin{proof}
See the previous proposition.	
\end{proof}

\begin{corollary}
The corresponding sheaf
\begin{center}
 $g_{\text{pro\'et},*} \mathbb{M}_{\mathrm{dR},X,A}\otimes_{\mathbb{B}_{\mathrm{dR},X,A}}	 \mathcal{O}\mathbb{B}_{\mathrm{dR},X,A}$
\end{center}
over $\mathcal{O}_X\widehat{\otimes}A$ is basically a coherent sheaf.	When we have that the corresponding space $X$ is smooth we could further have the corresponding projectivity in our context. 
\end{corollary}

\indent Now we consider the corresponding reconstruction process as in the philosophy of \cite{KP}. 

\begin{setting}
Namely we take the space $X$ to be smooth, and we take the corresponding $A$ to be some local chart of some abelian Fr\'echet-Stein algebra $A_\infty$ (but the local quasi-compacts are assumed to be within our sousperfectoid requirement). Take a $B$-pair $M$ now over $\mathcal{O}\mathbb{B}_{\mathrm{dR},X,A_\infty}$ namely let us require the corresponding finiteness and projectivity at the infinite level, as a family of $B$-pair over $\mathcal{O}\mathbb{B}_{\mathrm{dR},X,A_n}$ for all $n$.
\end{setting}

\begin{proposition}
Let $X$ be the space attached to $E\{T_1,...,T_n\}$, and define the functor $\mathbb{D}^\bullet_\mathrm{dR}$ as above for a family $M$ in the previous setting over $\mathcal{O}\mathbb{B}_{\mathrm{dR},X,A_\infty}$. Then we have that $\mathbb{D}^\bullet_\mathrm{dR}$ consists of coadmissible and finite projective modules over $\mathcal{O}_X\widehat{\otimes}A_\infty$ where $A_\infty$ is assumed to come from a $p$-adic Lie group as in \cite[Theorem 3.10]{Z1} (but the local quasi-compacts are assumed to be within our sousperfectoid requirement).
\end{proposition}

\begin{proof}
Our result above for $A_n$ for each $n=0,1,...$ shows the coadmissibility. The finite projectivity follows from the existence of the connection $\nabla_X$ and a theorem due to Z\'ab\'radi \cite[Theorem 3.10]{Z1}.	
\end{proof}

\begin{conjecture}
If the space $X$ is seminormal, then we have the same result.
\end{conjecture}

\begin{remark}
This would reconstruct, generalize and geometrize the corresponding Berger and Fourquaux's Lubin-Tate Iwasawa construction \cite{Berger}. And in the corresponding cyclotomic situation this also provides the corresponding generalization and geometrization of the corresponding construction after Nakamura \cite{Nakamura1}.	
\end{remark}

\newpage

\section{Applications to Equivariant Iwasawa Theory}

\subsection{Equivariant Big Perrin-Riou-Nakamura Exponential Maps} \label{section5.1}

\indent Now we establish the corresponding equivariant version of Nakamura's de Rham Iwasawa theory \cite{Nakamura1} of $(\varphi,\Gamma)$-modules over the Robba ring in the corresponding cyclotomic situation. The idea and basic philosophy come from the corresponding reconstruction process observed by the work of \cite{KPX} and \cite{KP}, namely the corresponding Fontaine's original construction using $\psi$-operator of the corresponding Iwasawa cohomology of $(\varphi,\Gamma)$-modules could be reconstructed by using the corresponding Iwasawa analytic deformation and even Frobenius sheaves in the style of Kedlaya-Liu \cite{KL15} and \cite{KL16} over the pro-\'etale sites after Scholze.\\

\indent In our previous work in \cite{T1}, we partially looked at some sort of generalized version of the picture above (although \cite{T1} is really and essentially aimed at the corresponding direction after \cite{PZ1} and \cite{CKZ}), namely instead of just deforming the corresponding $(\varphi,\Gamma)$-modules by some Iwasawa deformation over $\Pi^\infty(\Gamma)$ which is the corresponding Fr\'echet-Stein algebra attached to the group $\Gamma$, we look at more general Iwasawa deformation over $\Pi^\infty(\Gamma)\widehat{\otimes}\Pi^\infty(\Gamma')$ namely we could look at a suitable Fr\'echet-Stein algebra attached to the group $\Gamma\times \Gamma'$ such that we have that it could be written as family of affinoid algebras in rigid analytic geometry. The corresponding $(\varphi,\Gamma)$ cohomology of such modules could be regarded as the corresponding genuine $\Gamma'$-equivariant Iwasawa cohomology.\\

\indent Definitely we work in the philosophy of the corresponding Burns-Flach-Fukaya-Kato's grand picture on the so-called noncommutative Tamagawa number conjectures as in \cite{BF1}, \cite{BF2} and \cite{FK}. Nakamura's original Iwasawa theory on the $(\varphi,\Gamma)$-modules happens over the tower of Iwasawa tower. Namely the corresponding big Perrin-Riou-Nakamura exponential maps should be related to the corresponding characteristic ideals of the $(\varphi,\Gamma)$-cohomology of the corresponding Iwasawa deformation of any de Rham $(\varphi,\Gamma)$-module in some very precise way. One should regard our generalization as climbing even higher in some equivariant way. Namely the corresponding big Perrin-Riou-Nakamura exponential maps should be related to the corresponding characteristic ideals of the $(\varphi,\Gamma)$-cohomology of the corresponding higher dimensional equivariant deformation of any de Rham $(\varphi,\Gamma)$-module in some very precise way.\\

\indent We now follow \cite{Nakamura1} to establish the equivariant version of corresponding big (dual-)exponential maps after Perrin-Riou-Nakamura. This happens at the infinite level of the corresponding generalized admissible abelian tower in our setting, which generalizes the corresponding setting considered in \cite{Nakamura1}.


\begin{setting}
We consider an abelian $p$-adic Lie group $\Gamma'$ such that the Fr\'echet-Stein algebra $\Pi^\infty(\Gamma')$ in the sense of \cite{ST1} could be written as a family of affinoid algebras $\Pi^\infty(\Gamma')_k,k=0,1,...$ in the sense of rigid analytic geometry:
\begin{displaymath}
\Pi^\infty(\Gamma')\overset{\sim}{\longrightarrow} \varprojlim_k \Pi^\infty(\Gamma')_k.
\end{displaymath}
If we use the notation $X$ to denote the Fr\'echet-Stein space attached to the algebra $\Pi^\infty(\Gamma')$ which could be written as:
\begin{displaymath}
X\overset{\sim}{\longrightarrow} \varinjlim_k X_k:	
\end{displaymath}	
with:
\begin{displaymath}
\mathcal{O}_X\overset{\sim}{\longrightarrow} \varprojlim_k \mathcal{O}_{X_k}:	
\end{displaymath}
\end{setting}

\indent Now we start to consider the corresponding relative Robba rings where our Iwasawa theory happens:

\begin{definition}
Many definitions of the Robba rings are actually compatible with \cite{Nakamura1}, the only difference is that we will deform over some general algebras. We use the notations $\Pi^r_{\mathrm{an,con},K}$ and $\Pi_{\mathrm{an,con},K}$ to denote the corresponding Robba rings as in \cite[Section 2.1]{Nakamura1}. Now we consider the corresponding Robba rings $\Pi^r_{\mathrm{an,con},K,\Pi^\infty(\Gamma_K)\widehat{\otimes}\Pi^\infty(\Gamma')}$ and $\Pi_{\mathrm{an,con},K,\Pi^\infty(\Gamma_K)\widehat{\otimes}\Pi^\infty(\Gamma')}$ in the relative sense:
\begin{align}
\Pi^r_{\mathrm{an,con},K,\Pi^\infty(\Gamma_K)\widehat{\otimes}\Pi^\infty(\Gamma')}:=\Pi^r_{\mathrm{an,con},K}\widehat{\otimes}\Pi^\infty(\Gamma_K)\widehat{\otimes}\Pi^\infty(\Gamma'),\\
\Pi_{\mathrm{an,con},K,\Pi^\infty(\Gamma_K)\widehat{\otimes}\Pi^\infty(\Gamma')}:=\Pi_{\mathrm{an,con},K}\widehat{\otimes}\Pi^\infty(\Gamma_K)\widehat{\otimes}\Pi^\infty(\Gamma').\\	
\end{align}
One also has the corresponding definitions for Robba rings with respect to some interval $[s,r]$. And be very careful that in our current development the corresponding Robba ring $\Pi^?_{\mathrm{an,con},K}$ is the ring $\mathrm{B}^?_{\mathrm{rig},K}$ where $?=\empty,r,[s,r]$, instead of the ring $\Pi^?_{\mathrm{an,con}}(\pi_K)$ considered in \cite[Definition 2.2.2]{KPX}.	
\end{definition}

\begin{setting}
Our motivic objects are still usual $(\varphi,\Gamma)$-modules over the Robba ring $\Pi_{\mathrm{an,con},K}$. Namely we consider those modules finite locally free over this big non-noetherian ring carrying commuting semilinear action from the $\varphi$ and $\Gamma$. 	
\end{setting}

\begin{remark}
The ideas we presented here actually are initiated in \cite{KP}. One replaces the corresponding $p$-adic Lie group $\Gamma$ by a pro-\'etale site, and considers the deformation of the sheaves (both the corresponding motivic objects and the corresponding Berger's $p$-adic differential equations attached in whatever reasonable context like Lubin-Tate setting or some else) by the corresponding Iwasawa deformations, and then considers the corresponding big exponential maps by using the connections and their duals.	
\end{remark}

\indent To study Iwasawa theory we look at those geometric $(\varphi,\Gamma)$-modules, namely those geometric ones in the sense of $p$-adic Hodge theory after Berger's celebrated thesis. We do not go beyond \cite{Nakamura1} on this, but we would like to recall the basic construction on the corresponding $p$-adic differential equation attached to a de Rham $(\varphi,\Gamma)$-module $M$.

\begin{setting}
Recall the basic definition of the corresponding $p$-adic differential equation attache to a de Rham $(\varphi,\Gamma)$-module $M$ after Berger. As in \cite[Section 3.2]{Nakamura1} (we will use some different notations from \cite[Section 3.2]{Nakamura1}), we consider the following constructions. Let $M$ be a $(\varphi,\Gamma)$-module $M$ over $\Pi_{\mathrm{an,con},K}$. In our situation we work with at least de Rham ones in the sense of Berger's celebrated paper \cite{Ber1}. Then we have the corresponding relation for any integer $k\geq l'$ for some $l'(K,M)\geq 0$ fixed which is assumed to be sufficiently large depending on $K,M$:
\begin{align}
F_{\mathrm{dif},k}(M)\overset{\sim}{\longrightarrow} F_{\mathrm{dR},K}(M)\otimes K_k((t))	
\end{align}
with the corresponding lattice $F_{\mathrm{dR},K}(M)\otimes K_k[[t]]$ inside with reasonable rank. Then we have the corresponding continuous differential:
\begin{displaymath}
\widehat{\Omega}_{\Pi_{\mathrm{an,con},K}/K_0^\mathrm{ur}}=	\widehat{\Omega}_{\Pi_{\mathrm{an,con},K}}dT.
\end{displaymath}
 
\end{setting}

The corresponding differential equation in the sense of Berger is the following construction which is denoted by $\mathrm{N}_{\mathrm{dif}}(M)$, we will denote this by $\Theta_{\mathrm{Ber,dif}}(M)$ which is actually defined in the following way. Recall as in \cite[Theorem 3.5]{Nakamura1} by \cite[5.10]{Ber1} and \cite[3.2.3]{Ber2}, first we have $\Theta_{\mathrm{Ber,dif},k}(M)$ which is defined to be the subset of all the elements in $M^{(k)}[1/t]$ such that the corresponding image of them under each map $\iota_k'$ for all $k'\geq k$ lives in $F_{\mathrm{dR},K}(M)\otimes K_k'[[t]]$. Recall from \cite[Section 3.2]{Nakamura1} we have then the corresponding set $\Theta_{\mathrm{Ber,dif}}(M)=\varinjlim_k \Theta_{\mathrm{Ber,dif},k}(M)$ which could be endowed with furthermore the structure of $(\varphi,\Gamma)$-module over the Robba ring $\Pi_{\mathrm{an,con},K}$. And furthermore we have the corresponding properties that $d_0(\Theta_{\mathrm{Ber,dif}}(M))\subset t\Theta_{\mathrm{Ber,dif}}(M)$ where the differential operator $d_0$ is defined to be $\frac{\mathrm{log}(\gamma)}{\mathrm{log}(\chi(\gamma))}$ living in in our notation $\Pi^\infty(\Gamma_K)$.\\

\indent Therefore recall from \cite[Section 3.2]{Nakamura1} then we can define the connection $\partial:=\frac{1}{t}d_0$ mapping the corresponding differential $(\varphi,\Gamma)$-module considered above to itself. The differential structures and the corresponding Frobenius structures and the corresponding Lie group action structures are compatible in the following sense:
\begin{align}
\partial\varphi-p\varphi \partial=0,\\
\partial\gamma-\chi(\gamma)\gamma \partial=0,	
\end{align}
and:
\begin{align}
\varphi(dT)-pdT=0,\\
\gamma(dT)-\chi(\gamma)dT=0.	
\end{align}

\begin{lemma}
Now we set $d_i:=d_0-i$ for $i\geq 0$ integer, we have for any $m$-th ($m\geq 0$) Hodge filtration of $M$ (which is assumed to be de Rham) which gives the full module:
\begin{displaymath}
d_{m-1}d_{m-2}d_{m-3}...d_1d_0(\Theta_{\mathrm{Ber,dif}}(M))\subset M.	
\end{displaymath}
	
\end{lemma}

\begin{proof}
This is just \cite[Lemma 3.6]{Nakamura1}.	
\end{proof}

\indent We now define the corresponding equivariant big exponential map after Perrin-Riou-Nakamura.

\begin{assumption}
Now we assume that there is a quotient $G_K\rightarrow \Gamma'$ in our situation among all the other requirements as above.	
\end{assumption}

\begin{definition}
Take any sufficiently large $m\geq 0$ as in the previous lemma. First we consider the differential equation $\Theta_{\mathrm{Ber,dif}}(M)$ attached to $M$ which is assumed to be de Rham. We consider the $\Gamma'$-deformation of $M$ over the algebra $\Pi^\infty(\Gamma')$. This happens by considering the corresponding quotient $G_K\rightarrow \Gamma'$	which defines a corresponding Galois representation with value in $\Pi^\infty(\Gamma')$, which by corresponding embedding fully faithful gives rise to a corresponding $(\varphi,\Gamma)$-module over the relative Robba ring $\Pi_{\mathrm{an,con},K}\widehat{\otimes}\Pi^\infty(\Gamma')$. Denote this by $\mathrm{Dfm}_{\Gamma'}$ which could be written as the projective limit of the families $\{\mathrm{Dfm}_{\Gamma',p}\}_{p\geq 0}$ over $\Pi_{\mathrm{an,con},K}\widehat{\otimes}\Pi^\infty(\Gamma')_p$ for all $p\geq 0$. Then we consider the product $\mathrm{Dfm}_{\Gamma',p}\widehat{\otimes} M$ then take the corresponding projective limit over all $p\geq 0$. Then we consider the following definition which is an equivariant version of \cite[Definition 3.7]{Nakamura1}:
\begin{displaymath}
\mathrm{Exp}^{\Pi^\infty(\Gamma')}_{M,m}: (\varprojlim \mathrm{Dfm}_{\Gamma',p}\widehat{\otimes} \Theta_{\mathrm{Ber,dif}}(M)	)^{\psi=1}\rightarrow (\varprojlim \mathrm{Dfm}_{\Gamma',p}\widehat{\otimes} M	)^{\psi=1}
\end{displaymath}
which is just defined by extension of $d_{m-1}d_{m-2}d_{m-3}...d_1d_0$. This is now $\Pi^\infty(\Gamma_K)\widehat{\otimes}\Pi^\infty(\Gamma')$-linear instead of just $\Pi^\infty(\Gamma_K)$. And the corresponding $(\varphi,\Gamma)$-modules here could be basically geometrized by the corresponding Hodge-Iwasawa construction above. We consider the corresponding family of vector bundle $\varprojlim\mathcal{F}(M)_p$ and $\varprojlim\mathcal{F}(\Theta_{\mathrm{Ber,dif}}(M))_p$ over the corresponding adic version of Fargues-Fontaine curve deformed by $\Pi^\infty(\Gamma_K)\widehat{\otimes}\Pi^\infty(\Gamma')$. We have now that the geometrized big exponential map by taking the corresponding sheaf cohomology:
\begin{displaymath}
\mathrm{Exp}^{\Pi^\infty(\Gamma')}_{\mathcal{F},m}: R^1\Gamma_\mathrm{sheaf}(\varprojlim\mathcal{F}(\Theta_{\mathrm{Ber,dif}}(M))_p) \rightarrow R^1\Gamma_\mathrm{sheaf}(\varprojlim\mathcal{F}(M)_p),
\end{displaymath}
and the corresponding derived version:
\begin{displaymath}
\mathrm{Exp}^{\Pi^\infty(\Gamma'),\bullet}_{\mathcal{F},m}: R^\bullet\Gamma_\mathrm{sheaf}(\varprojlim\mathcal{F}(\Theta_{\mathrm{Ber,dif}}(M))_p) \rightarrow R^\bullet\Gamma_\mathrm{sheaf}(\varprojlim\mathcal{F}(M)_p).
\end{displaymath}\\
\end{definition}

\subsection{Characteristic Ideal Sheaves} \label{section5.2}

\indent Now we follow the corresponding idea in \cite{Nakamura1} to consider the corresponding characteristic ideal and the corresponding determinant of a given derived maps with respect to the corresponding equivariant Iwasawa cohomology. For this we consider the corresponding map defined above, and we look at the corresponding equivariant Iwasawa cohomology, i.e. the corresponding more general commutative deformation in our current context.\\

\indent First in our situation we make the following assumption:

\begin{assumption}
We now assume that the corresponding deformation happens over Krull noetherian rings. This could basically allow us to define the corresponding characteristic ideal over the sheaf $\Pi^\infty(\Gamma_K)\widehat{\otimes} X_{\Pi^\infty(\Gamma')}$ of any torsion module $M$, which we will denote it by $\mathrm{char}_{\Pi^\infty(\Gamma_K)\widehat{\otimes}X_{\Pi^\infty(\Gamma')}}(M)$ which is again a sheaf of module over $\Pi^\infty(\Gamma_K)\widehat{\otimes} X_{\Pi^\infty(\Gamma')}$.	
\end{assumption}

\indent In the current generalized context, following Nakamura \cite[Section 3.4]{Nakamura1} we define the corresponding determinant of a general map $f:M\rightarrow N$ of coadmissible modules over the sheaf $\Pi^\infty(\Gamma_K)\widehat{\otimes}X_{\Pi^\infty(\Gamma')}$. First as above we have the corresponding characteristic ideals $\mathrm{char}_{\Pi^\infty(\Gamma_K)\widehat{\otimes}X_{\Pi^\infty(\Gamma')}}(M_\mathrm{torsion})$ and $\mathrm{char}_{\Pi^\infty(\Gamma_K)\widehat{\otimes}X_{\Pi^\infty(\Gamma')}}(N_\mathrm{torsion})$. Then quotienting out the corresponding torsion part we have:
\begin{align}
\overline{f}:M/M_{\mathrm{torsion}}\rightarrow 	N/N_{\mathrm{torsion}}.
\end{align}

We will regard this as a corresponding map on the sheaves over the the space $\Pi^\infty(\Gamma_K)\widehat{\otimes}X_{\Pi^\infty(\Gamma')}$, then for any $Y=\mathrm{Sp}(\Pi^\infty(\Gamma')_p)$ quasicompact of $X_{\Pi^\infty(\Gamma')}$ we consider the corresponding ring $\Pi^\infty(\Gamma_K)\widehat{\otimes}\Pi^\infty(\Gamma')_p$, then we have the notion of determinant of $\overline{f}$ which we will denote it by $\mathrm{det}_{\Pi^\infty(\Gamma_K)\widehat{\otimes}\Pi^\infty(\Gamma')_p}(\overline{f})$. We then organize this construction to get the determinant sheaf $\mathrm{det}_{\Pi^\infty(\Gamma_K)\widehat{\otimes}X_{\Pi^\infty(\Gamma')}}(\overline{f})$. Then now we define the determinant of $f$ as:
\begin{align}
\mathrm{det}_{\Pi^\infty(\Gamma_K)\widehat{\otimes}X_{\Pi^\infty(\Gamma')}}&(f)\\
&:=\mathrm{det}_{\Pi^\infty(\Gamma_K)\widehat{\otimes}X_{\Pi^\infty(\Gamma')}}(\overline{f})\mathrm{char}_{\Pi^\infty(\Gamma_K)\widehat{\otimes}X_{\Pi^\infty(\Gamma')}}(M_\mathrm{torsion})\mathrm{char}_{\Pi^\infty(\Gamma_K)\widehat{\otimes}X_{\Pi^\infty(\Gamma')}}(N_\mathrm{torsion})^{-1}.	
\end{align}

\indent On the corresponding derived level we consider after \cite[Section 3.4]{Nakamura1} the following construction for the big exponential map:
\begin{displaymath}
\mathrm{Exp}^{\Pi^\infty(\Gamma'),\bullet}_{\mathcal{F},m}: R^\bullet\Gamma_\mathrm{sheaf}(\varprojlim\mathcal{F}(\Theta_{\mathrm{Ber,dif}}(M))_p) \rightarrow R^\bullet\Gamma_\mathrm{sheaf}(\varprojlim\mathcal{F}(M)_p).
\end{displaymath}

For each $p$ we consider the localized map:

\begin{displaymath}
\mathrm{Exp}^{\Pi^\infty(\Gamma')_p,\bullet}_{\mathcal{F},m}: R^\bullet\Gamma_\mathrm{sheaf}(\mathcal{F}(\Theta_{\mathrm{Ber,dif}}(M))_p) \rightarrow R^\bullet\Gamma_\mathrm{sheaf}(\mathcal{F}(M)_p).
\end{displaymath}

Then what we can do is to define the corresponding determinant of the this by:
\begin{align}
\mathrm{det}_{\Pi^\infty(\Gamma_K)\widehat{\otimes}\Pi^\infty(\Gamma')_p}&(\mathrm{Exp}^{\Pi^\infty(\Gamma')_p,\bullet}_{\mathcal{F},m}):=	\\
& \mathrm{det}_{\Pi^\infty(\Gamma_K)\widehat{\otimes}\Pi^\infty(\Gamma')_p}(R^1\Gamma_\mathrm{sheaf}(\mathcal{F}(\Theta_{\mathrm{Ber,dif}}(M))_p) \rightarrow R^1\Gamma_\mathrm{sheaf}(\mathcal{F}(M)_p))\cdot\\
&\mathrm{det}_{\Pi^\infty(\Gamma_K)\widehat{\otimes}\Pi^\infty(\Gamma')_p}(R^2\Gamma_\mathrm{sheaf}(\mathcal{F}(\Theta_{\mathrm{Ber,dif}}(M))_p) \rightarrow R^2\Gamma_\mathrm{sheaf}(\mathcal{F}(M)_p))^{-1}.
\end{align}

Then we can organize this to be the corresponding determinant sheaf:

\begin{align}
\mathrm{det}_{\Pi^\infty(\Gamma_K)\widehat{\otimes}\Pi^\infty(\Gamma')}&(\mathrm{Exp}^{X_{\Pi^\infty(\Gamma')},\bullet}_{\mathcal{F},m})=	\\
& \mathrm{det}_{\Pi^\infty(\Gamma_K)\widehat{\otimes}\Pi^\infty(\Gamma')}(R^1\Gamma_\mathrm{sheaf}(\mathcal{F}(\{\Theta_{\mathrm{Ber,dif}}(M))_p\}_p) \rightarrow R^1\Gamma_\mathrm{sheaf}(\{\mathcal{F}(M)_p\}_p))\cdot\\
&\mathrm{det}_{\Pi^\infty(\Gamma_K)\widehat{\otimes}\Pi^\infty(\Gamma')}(R^2\Gamma_\mathrm{sheaf}(\mathcal{F}(\{\Theta_{\mathrm{Ber,dif}}(M))_p\}_p) \rightarrow R^2\Gamma_\mathrm{sheaf}(\{\mathcal{F}(M)_p\}_p))^{-1}\\
&= \mathrm{det}_{\Pi^\infty(\Gamma_K)\widehat{\otimes}\Pi^\infty(\Gamma')}(R^1\Gamma_\mathrm{sheaf}(\mathcal{F}(\{\Theta_{\mathrm{Ber,dif}}(M))_p\}_p) \rightarrow R^1\Gamma_\mathrm{sheaf}(\{\mathcal{F}(M)_p\}_p))\cdot\\
&\mathrm{char}_{\Pi^\infty(\Gamma_K)\widehat{\otimes}\Pi^\infty(\Gamma')}(R^2\Gamma_\mathrm{sheaf}(\mathcal{F}(\{\Theta_{\mathrm{Ber,dif}}(M))_p\}_p))^{-1}\cdot\\
&\mathrm{char}_{\Pi^\infty(\Gamma_K)\widehat{\otimes}\Pi^\infty(\Gamma')}(R^2\Gamma_\mathrm{sheaf}(\{\mathcal{F}(M)_p\}_p)). 
\end{align}\\

\indent The following conjecture reads that in our current situation we have the corresponding equivariant Iwasawa main conjecture. But we will only consider the situation where $K=\mathbb{Q}_p$ and the corresponding full admissible $p$-adic Lie extension after \cite{LZ1} namely we consider $K_\infty=\mathbb{Q}^\mathrm{ur}_p(\mu_{p^\infty})$. We first recall Nakamura's result in \cite{Nakamura1}:

\begin{theorem} \mbox{\bf{(Nakamura's Iwasawa Main Conjecture)}}
In our current notations, let $\Gamma'$ be trivial. Consider a de Rham $(\varphi,\Gamma)$ module $M$ over the usual Robba ring. We assume that this module has rank $g$ with $g$ Hodge-Tate weights	$a_1,...a_g$. Then we have the following formula linking the determinant of the Perrin-Riou-Nakamura big exponential map and the characteristic ideals in the equivariant way:
\begin{align}
\prod_{i}\prod^{g-a_i-1}_{b_i} \nabla_{a_i+b_i} = \mathrm{det}_{\Pi^\infty(\Gamma_K)\widehat{\otimes}\Pi^\infty(\Gamma')}&(R^1\Gamma_\mathrm{sheaf}(\mathcal{F}(\{\Theta_{\mathrm{Ber,dif}}(M))_p\}_p) \rightarrow R^1\Gamma_\mathrm{sheaf}(\{\mathcal{F}(M)_p\}_p))\cdot\\
&\mathrm{char}_{\Pi^\infty(\Gamma_K)\widehat{\otimes}\Pi^\infty(\Gamma')}(R^2\Gamma_\mathrm{sheaf}(\mathcal{F}(\{\Theta_{\mathrm{Ber,dif}}(M))_p\}_p))^{-1}\cdot\\
&\mathrm{char}_{\Pi^\infty(\Gamma_K)\widehat{\otimes}\Pi^\infty(\Gamma')}(R^2\Gamma_\mathrm{sheaf}(\{\mathcal{F}(M)_p\}_p)).	
\end{align}
\end{theorem}

Then in this situation:

\begin{conjecture} \mbox{\bf{(Equivariant Iwasawa Main Conjecture, after Nakamura)}}
Consider a de Rham $(\varphi,\Gamma)$ module $M$ over the usual Robba ring. We assume that this module has rank $g$ with $g$ Hodge-Tate weights	$a_1,...a_g$. Then we conjecture we have the corresponding analog of the formula linking the determinant of the Perrin-Riou-Nakamura big exponential map and the characteristic ideals in the equivariant way.
\end{conjecture}

\begin{remark}
This is also inspired by the work of \cite{LVZ} and \cite{LZ1}, while we work in the generality after \cite{Nakamura1}. Although we have not mentioned explicitly but we definitely conjecture that one has the parallel statement for Lubin-Tate differential equations after \cite{Berger} in the equivariant setting. 	
\end{remark}

\newpage

\subsection*{Acknowledgements} 

The paper grew along our study on the Hodge-Iwasawa theory initiated closely after \cite{KP}, \cite{KL15} and \cite{KL16}, which benefits from different perspectives from the ideas and the inspiration from the essential innovations included in these previous works. We would like to thank Professor Kedlaya for the corresponding discussion along our preparation in different periods on deep foundations and various possible extensions of \cite{KP}, \cite{KL15} and \cite{KL16}.

\newpage

\bibliographystyle{ams}

\end{document}